\documentclass[letterpaper, 11pt,  reqno]{amsart}

\usepackage{amsmath,amssymb,amscd,amsthm,amsxtra, esint}
\usepackage{color, colortbl}
\usepackage{tikz} 

\usepackage{comment}
\usepackage{mathtools}
\usepackage{color}
\usepackage[implicit=true]{hyperref}
\usepackage{scalerel} 

\usepackage{cases}

\usepackage[margin=1.2in,marginparwidth=1.5cm, marginparsep=0.5cm]{geometry}

\allowdisplaybreaks[2]

\usepackage{color}

\definecolor{gr}{rgb}   {0.,   0.69,   0.23 }
\definecolor{bl}{rgb}   {0.,   0.5,   1. }
\definecolor{mg}{rgb}   {0.85,  0.,    0.85}
\definecolor{gy}{rgb}   {0.9,  0.9,   0.9}
\definecolor{yl}{rgb}   {0.8,  0.7,   0.}
\definecolor{or}{rgb}  {0.7,0.2,0.2}

\usetikzlibrary{shapes.misc}
\usetikzlibrary{shapes.symbols}
\usetikzlibrary{shapes.geometric}

\tikzset{
	ddot/.style={circle,fill=white,draw=black,inner sep=0pt,minimum size=0.8mm},
	>=stealth,
	}

\tikzset{
	ddot2/.style={circle,fill=black,draw=black,inner sep=0pt,minimum size=0.8mm},
	>=stealth,
	}

\newtheorem{theorem}{Theorem} [section]

\newtheorem{lemma}[theorem]{Lemma}
\newtheorem{proposition}[theorem]{Proposition}
\newtheorem{remark}[theorem]{Remark}
\newtheorem{example}{Example}

\DeclareMathOperator*{\supp}{supp}

\newcommand{\noi}{\noindent}
\newcommand{\Z}{\mathbb{Z}}
\newcommand{\R}{\mathbb{R}}

\newcommand{\T}{\mathbb{T}}

\newcommand{\Pf}{\mathfrak{P}}

\newcommand{\F}{\mathcal{F}}

\newcommand{\dl}{\delta}

\newcommand{\Dl}{\Delta}
\newcommand{\eps}{\varepsilon}

\newcommand{\G}{\Gamma}

\newcommand{\ft}{\widehat}

\newcommand{\wt}{\widetilde}
\newcommand{\cj}{\overline}

\newcommand{\dt}{\partial_t}

\newcommand{\embeds}{\hookrightarrow}

\newcommand{\ta}{\theta}

\renewcommand{\l}{\ell}

\newcommand{\les}{\lesssim}
\newcommand{\ges}{\gtrsim}

\newcommand{\jb}[1]
{\langle #1 \rangle}

\newcommand{\ind}{\mathbf 1}

\newcommand{\N}{\mathbb{N}}

\newtheorem*{ackno}{Acknowledgements}

\numberwithin{equation}{section}
\numberwithin{theorem}{section}

\makeatletter
\@namedef{subjclassname@2020}{%
  \textup{2020} Mathematics Subject Classification}
\makeatother

\begin{document}
\baselineskip = 13.5pt

\title[LWP of the periodic quadratic NLS in negative Sobolev spaces]
{Local well-posedness of the periodic nonlinear Schr\"odinger equation with a quadratic nonlinearity $\cj{u}^2$ in negative Sobolev spaces}

\author[R.~Liu]
{Ruoyuan Liu}

\address{
Ruoyuan Liu,  School of Mathematics\\
The University of Edinburgh\\
and The Maxwell Institute for the Mathematical Sciences\\
James Clerk Maxwell Building\\
The King's Buildings\\
Peter Guthrie Tait Road\\
Edinburgh\\ 
EH9 3FD\\
 United Kingdom}

\email{ruoyuan.liu@ed.ac.uk}


%


\begin{abstract}
We study low regularity local well-posedness of the nonlinear Schr\"odinger equation (NLS) with the quadratic nonlinearity~$\cj{u}^2$, posed on one-dimensional and two-dimensional tori. While the relevant bilinear estimate with respect to the $X^{s, b}$-space is known to fail when the regularity $s$ is below some threshold value, we establish local well-posedness for such low regularity by introducing modifications on the $X^{s, b}$-space.

\end{abstract}

\maketitle

%

\section{Introduction}
\label{SEC:intro}

\subsection{Quadratic nonlinear Schr\"odinger equations}
In this paper, we consider the following Cauchy problem for the quadratic nonlinear Schr\"odinger equation (NLS) on periodic domains:
\begin{equation}
\begin{cases}
i \dt u + \Dl u = \cj{u}^2 \\
u|_{t = 0} = u_0
\end{cases}
\quad (x, t) \in \mathcal{M} \times \R,
\label{qNLS}
\end{equation}

\noi
where $\mathcal{M} = \T$ or $\T^2$ with $\T = \R / 2 \pi \Z$.

Our main goal is to establish low regularity local well-posedness of the quadratic NLS \eqref{qNLS} on periodic domains $\T$ or $\T^2$. For instructive purposes, we first provide some background on the quadratic NLS
\begin{align}
i \dt u + \Dl u = \mathcal{N} (u, u),
\label{qNLS_gen}
\end{align}

\noi
where $\mathcal{N} (u, u)$ can be $u^2$, $\cj{u}^2$, or $|u|^2$. Note that on $\R^d$, if $u$ is a solution to \eqref{qNLS_gen}, then $u_\lambda (x, t) := \lambda^2 u (\lambda x, \lambda^2 t)$ is also a solution to \eqref{qNLS_gen} for any $\lambda > 0$. This scaling symmetry induces the following scaling critical Sobolev regularity:
\begin{align*}
s_{\text{crit}} = \frac{d}{2} - 2.
\end{align*}

\noi
When $d \leq 3$, the scaling critical regularity is negative, which often fails to predict well-posedness and ill-posedness issues. In this paper, we mainly focus on the cases when $d = 1$ and $d = 2$.

Let us now review some previous results on the quadratic NLS \eqref{qNLS_gen}, starting with the real line case. In \cite{KPV96}, Kenig-Ponce-Vega used the Bourgain space $X^{s, b}$ (see Subsection \ref{SUBSEC:Xsb}) to prove local well-posedness of \eqref{qNLS_gen} on $\R$ for all types of nonlinearities $u^2$, $\cj{u}^2$, and $|u|^2$. Specifically, they established the following bilinear estimates:
\begin{align}
\| u v \|_{X^{s, b - 1}} &\les \| u \|_{X^{s, b}} \| v \|_{X^{s, b}}, \label{bi1} \\
\| \cj{u} \cj{v} \|_{X^{s, b - 1}} &\les \| u \|_{X^{s, b}} \| v \|_{X^{s, b}} \label{bi2}
\end{align}
for $s > -\frac 34$ and $b = \frac 12 +$ and\footnote{Here, $\frac 12 +$ means $\frac 12 + \eps$ for some $\eps > 0$.}
\begin{align}
\| u \cj{v} \|_{X^{s, b - 1}} &\les \| u \|_{X^{s, b}} \| v \|_{X^{s, b}} \label{bi3}
\end{align}
for $s > -\frac 14$ and $b = \frac 12 +$. In addition, in the same paper, they showed that \eqref{bi1} and \eqref{bi2} fail for $s < -\frac 34$ and \eqref{bi3} fails for $s < -\frac 14$. The failure of these bilinear estimates at the endpoint regularities were established in \cite{NTT}. Despite the failure of the bilinear estimate \eqref{bi1}, Bejenaru-Tao \cite{BT} showed local well-posedness of \eqref{qNLS_gen} on $\R$ with nonlinearity $\mathcal{N} (u, u) = u^2$ for $s \geq -1$ by introducing weighted function spaces. Moreover, they proved ill-posedness of the same equation for $s < -1$. Later, Kishimoto \cite{Kish08} proved local well-posedness of \eqref{qNLS_gen} on $\R$ with $\mathcal{N} (u, u) = \cj{u}^2$ for $s \geq -1$ using different weighted function spaces. He also proved ill-posedness of the same equation for $s < -1$. Regarding \eqref{qNLS_gen} on $\R$ with $\mathcal{N} (u, u) = |u|^2$, Kishimoto \cite{Kishm} showed local well-posedness for $s \geq -\frac 14$ and ill-posedness for $s < -\frac 14$ (see also \cite{KT}). See also \cite{IO, IU, Kish19} for stronger ill-posedness results in the same ranges of $s$. For convenience, we summarize these results in Table \ref{table1}. Note that for all these nonlinearities $u^2$, $\cj{u}^2$, and $|u|^2$, well-posedness and ill-posedness results are sharp. Also, for all these nonlinearities, ill-posedness occurs before $s$ reaches the scaling critical regularity of \eqref{qNLS_gen} on $\R$: $s_{\text{crit}} = -\frac 32$.

\begin{table} \centering
\arrayrulewidth = 1pt
\renewcommand{\arraystretch}{1.5}
\begin{tabular}[t]{| c | c | c | c | c | c | c |}
\hline
Setting & \multicolumn{3}{c |}{$\R$} & \multicolumn{3}{ c |}{$\R^2$} \\
\hline
Nonlinearity $\mathcal{N} (u, u)$ & $u^2$ & $\cj{u}^2$ & $|u|^2$ & $u^2$ & $\cj{u}^2$ & $|u|^2$ \\
\hline
Scaling critical regularity & \multicolumn{3}{c |}{$s_{\text{crit}} = - \frac 32$} & \multicolumn{3}{ c |}{$s_{\text{crit}} = - 1$} \\
\hline 
$X^{s, b}$-bilinear estimate & $s > - \frac 34$ & $s > - \frac 34$ & $s > - \frac 14$ & $s > -\frac 34$ & $s > - \frac 34$ & $s > -\frac 14$ \\
\hline
Failure of $X^{s, b}$-bilinear estimate & $s \leq - \frac 34$ & $s \leq - \frac 34$ & $s \leq - \frac 14$ & $s \leq - \frac 34$ & $s \leq - \frac 34$ & $s \leq - \frac 14$ \\
\hline
Local well-posedness & $s \geq -1$ & $s \geq -1$ & $s \geq - \frac 14$ & $s > -1$ & $s > -1$ & $s \geq - \frac 14$ \\
\hline
Ill-posedness & $s < -1$ & $s < -1$ & $s < -\frac 14$ & $s \leq -1$ & $s \leq -1$ & $s < -\frac 14$ \\
\hline
\end{tabular}
\vspace{8pt}
\caption{Known results for the quadratic NLS on $\R$ and $\R^2$. Note that in all cases, local well-posedness and ill-posedness results are sharp.}
\label{table1}
\end{table}

Let us also mention well-posedness and ill-posedness results of \eqref{qNLS_gen} on $\R^2$, which are again summarized in Table \ref{table1}. The $X^{s, b}$-bilinear estimates \eqref{bi1}, \eqref{bi2}, and \eqref{bi3} were established in \cite{CDKS, Sta97}. The failure of these $X^{s, b}$-bilinear estimates for lower values of $s$ was shown in \cite{CDKS, NTT}. For local well-posedness of \eqref{qNLS_gen} on $\R^2$, see \cite{BS, Kishm, Kish19}. For ill-posedness of \eqref{qNLS_gen} on $\R^2$, see \cite{IO, IU, Kish19}. From Table \ref{table1}, we note that the ill-posedness on $\R^2$ for $\mathcal{N} (u, u) = |u|^2$ occurs before $s$ reaches the scaling critical regularity $s_{\text{crit}} = -1$. Also, we can see that all well-posedness and ill-posedness results are sharp on $\R^2$.

We now turn our attention to well-posedness and ill-posedness results of \eqref{qNLS_gen} on periodic domains $\T$ and $\T^2$. The results are summarized in Table \ref{table2}. On $\T$, for all nonlinearities $u^2$, $\cj{u}^2$, and $|u|^2$, the $X^{s, b}$-bilinear estimates \eqref{bi1}, \eqref{bi2}, and \eqref{bi3} for $s \geq 0$ follows immediately from the $L^3$-Strichartz estimate, which is obtained by interpolating the $L^4$-Strichartz estimate on $\T$ (see \cite{Bour93, Zyg}) and the trivial $L^2$-bound. In \cite{KPV96}, Kenig-Ponce-Vega established bilinear estimates \eqref{bi1} (for $u^2$) and \eqref{bi2} (for $\cj{u}^2$) on $\T$ for $s > -\frac 12$ and $b = \frac 12 +$ and showed corresponding local well-posedness results. They also showed that \eqref{bi1} and \eqref{bi2} fail on $\T$ when $s < -\frac 12$ and \eqref{bi3} (for $|u|^2$) fails on $\T$ when $s < 0$. Later, Kishimoto \cite{Kish19} showed ill-posedness of \eqref{qNLS_gen} on $\T$ with all types of nonlinearities for regularity ranges shown in Table \ref{table2}. Here, we note that there are gaps between local well-posedness and ill-posedness results for nonlinearities $u^2$ and $\cj{u}^2$. Also, the quadratic NLS \eqref{qNLS_gen} with nonlinearity $|u|^2$ behaves worse on $\T$ than on $\R$, since ill-posedness on $\T$ occurs for a wider range of $s$ than on $\R$.

For \eqref{qNLS_gen} on $\T^2$ with all nonlinearities $u^2$, $\cj{u}^2$, and $|u|^2$, the $X^{s, b}$-bilinear estimates \eqref{bi1}, \eqref{bi2}, and \eqref{bi3} for $s > 0$ follows from the $L^3$-Strichartz estimate with an $\eps$ derivative loss, which is obtained by interpolating the $L^4$-Strichartz estimate on $\T^2$ (see Lemma \ref{LEM:L4}) and the trivial $L^2$-bound. In \cite{Grun}, Gr\"unrock showed the bilinear estimate \eqref{bi2} (for $\cj{u}^2$) for $s > -\frac 12$ and proved the corresponding local well-posedness result. In the same paper, he showed the failure of \eqref{bi1} (for $u^2$) on $\T^2$ when $s < 0$ and the failure of \eqref{bi2} (for $\cj{u}^2$) on $\T^2$ when $s < -\frac 12$. In \cite{Kish19}, Kishimoto showed ill-posedness of \eqref{qNLS_gen} on $\T^2$ with all types of nonlinearities for regularity ranges shown in Table \ref{table2}. In a recent work, Oh and the author \cite{LO} proved local well-posedness of \eqref{qNLS_gen} with nonlinearities $u^2$ and $|u|^2$ for $s = 0$ by establishing correponding $X^{s, b}$-bilinear estimates.

\begin{table} \centering
\arrayrulewidth = 1pt
\renewcommand{\arraystretch}{1.5}
\begin{tabular}[t]{| c | c | c | c | c | c | c |}
\hline
Setting & \multicolumn{3}{c |}{$\T$} & \multicolumn{3}{ c |}{$\T^2$} \\
\hline
Nonlinearity $\mathcal{N} (u, u)$ & $u^2$ & $\cj{u}^2$ & $|u|^2$ & $u^2$ & $\cj{u}^2$ & $|u|^2$ \\
\hline
Scaling critical regularity & \multicolumn{3}{c |}{$s_{\text{crit}} = - \frac 32$} & \multicolumn{3}{ c |}{$s_{\text{crit}} = - 1$} \\
\hline 
$X^{s, b}$-bilinear estimate & $s > - \frac 12$ & $s > - \frac 12$ & $s \geq 0$ & $s \geq 0$ & $s > - \frac 12$ & $s \geq 0$ \\
\hline
Failure of $X^{s, b}$-bilinear estimate & $s < - \frac 12$ & $s < - \frac 12$ & $s < 0$ & $s < 0$ & $s < - \frac 12$ & $s < 0$ \\
\hline
Local well-posedness & $s > - \frac 12$ & $\pmb{s > - \frac 23}$ & $s \geq 0$ & $s \geq 0$ & $\pmb{s > - \frac 23}$ & $s \geq 0$ \\
\hline
Ill-posedness & $s < -1$ & $s < -1$ & $s < 0$ & $s \leq -1$ & $s \leq -1$ & $s < 0$ \\
\hline
\end{tabular}
\vspace{8pt}
\caption{Currect results for the quadratic NLS on $\T$ and $\T^2$. The boldface texts in the table refer to the results of this paper. Note that for the nonlinearity $|u|^2$, local well-posedness and ill-posedness results are sharp on both $\T$ and $\T^2$. For nonlinearities $u^2$ and $\cj{u}^2$ on either $\T$ or $\T^2$, there are gaps between local well-posedness and ill-posedness results.}
\label{table2}
\end{table}

The long-time behaviors of the quadratic NLS \eqref{qNLS_gen} have also been studied. For global existence and scattering results, see \cite{FG, GMS, HNST, JLT, MTT, ST}. For nonexistence of non-trivial scattering solutions, see \cite{Shi, STsu}. For finite-time blowup results, see \cite{IW, Oh}.

\medskip
As can be seen from Table \ref{table2}, local well-posedness for the quadratic NLS with nonlinearity $|u|^2$ is complete, whereas for nonlinearities $u^2$ and $\cj{u}^2$, there are gaps between local well-posedness and ill-posedness results. The difference of well-posedness behaviors of these three nonlinearities is closed related to their distinct phase functions. By letting $n_1, n_2$ be the frequencies of the nonlinearity and $n$ be the frequency of the duality term, we can write out the frequency interactions and phase functions for these three nonlinearities as in Table \ref{table3}.

When the phase function is large, we expect some gain of regularities. For example, for nonlinearity $\cj{u}^2$ on $\T^2$, the phase function $|n|^2 + |n_1|^2 + |n_2|^2$ provides gain of derivatives, so that one can establish local well-posedness for nonlinearity $\cj{u}^2$ with very rough initial data. On the other hand, for nonlinearity $u^2$ on $\T^2$, the phase function $|n|^2 - |n_1|^2 - |n_2|^2 = 2 n_1 \cdot n_2$ can be very small if $n_1$ and $n_2$ are almost perpendicular to each other, so that local well-posedness with rough initial data is much harder. In this paper, we focus on shrinking the well-posedness gap for nonlinearity $\cj{u}^2$ by establishing local well-posedness with lower regularity. We also discuss some well-posedness issues for nonlinearity $u^2$ in Remark \ref{RMK:uu} below.

\begin{table} \centering
\arrayrulewidth = 1pt
\renewcommand{\arraystretch}{1.5}
\begin{tabular}[t]{| c | c | c | c |}
\hline
Nonlinearity $\mathcal{N} (u, u)$ & $u^2$ & $\cj{u}^2$ & $|u|^2$ \\
\hline
Frequency interaction & $n - n_1 - n_2 = 0$ & $n + n_1 + n_2 = 0$ & $n - n_1 + n_2 = 0$ \\
\hline 
Phase function & $|n|^2 - |n_1|^2 - |n_2|^2$ & $|n|^2 + |n_1|^2 + |n_2|^2$ & $|n|^2 - |n_1|^2 + |n_2|^2$ \\
\hline
\end{tabular}
\vspace{8pt}
\caption{Frequency interactions and phase functions for the quadratic NLS with nonlinearities $u^2$, $\cj{u}^2$, and $|u|^2$.}
\label{table3}
\end{table}

\medskip
We now look back on low regularity local well-posedness of the quadratic NLS \eqref{qNLS} on $\T$ and $\T^2$. In this paper, we prove the following theorem.

\begin{theorem}
\label{THM:LWP}
Let $\mathcal{M} = \T$ or $\T^2$. Then, the quadratic NLS \eqref{qNLS} is locally well-posed in $H^s (\mathcal{M})$ for $s > -\frac 23$. More precisely, given any $u_0 \in H^s (\mathcal{M})$, there exists $T = T(\| u_0 \|_{H^s}) > 0$ and a unique solution $u \in C([-T, T]; H^s (\mathcal{M}))$ to \eqref{qNLS} with $u|_{t = 0} = u_0$, and the solution $u$ depends continuously on the initial data $u_0$.
\end{theorem}

Since local well-posedness of \eqref{qNLS} for $s > -\frac 12$ was already shown in \cite{KPV96} on $\T$ and in \cite{Grun} on $\T^2$, we mainly focus on the situation when $-\frac 23 < s \leq -\frac 12$. Our proof of Theorem \ref{THM:LWP} relies on modified $X^{s,b}$-spaces for the solutions, and so the uniqueness in the above statement holds only in the relevant function space (see the $Z^{s, b}$-norm in \eqref{Zsb} and its local-in-time version in \eqref{ZsbT}). For the proof of Theorem \ref{THM:LWP}, we will mainly focus on the case $\mathcal{M} = \T^2$ (see Remark \ref{RMK:T}). The idea of the proof of Theorem \ref{THM:LWP} is to introduce modifications on the $X^{s,b}$-space which enable us to prove the corresponding bilinear estimate. See Subsection \ref{SUBSEC:BT} for more discussion on it.

Theorem \ref{THM:LWP} improves the previous local well-posedness results in \cite{Grun, KPV96}. In addition, to the best of the author's knowledge, these are the first local well-posedness results for the quadratic NLS on periodic domains below the regularity thresholds where the usual $X^{s, b}$-bilinear estimates fail. We also remark that the bound $s > -\frac 23$ is sharp (up to the endpoint regularity $s = -\frac 23$) in our approach. See Subsection \ref{SUBSEC:BT} for more details.

\begin{remark} \rm
\label{RMK:T}
In Theorem \ref{THM:LWP}, the proof for $\mathcal{M} = \T$ follows from the proof for $\mathcal{M} = \T^2$ with minor modifications. Thus, in proving Theorem \ref{THM:LWP}, we mainly restrict our attention on the case $\mathcal{M} = \T^2$.
\end{remark}

\subsection{Modified function spaces}
\label{SUBSEC:BT}

In this subsection, we briefly explain our strategy for proving Theorem \ref{THM:LWP}.

In \cite{BT}, Bejenaru-Tao reduced the well-posedness problem of the quadratic NLS \eqref{qNLS_gen} in $H^s (\R^d)$ or $H^s (\T^d)$ to finding a space-time norm $\| \cdot \|_{W^s}$ that satisfy the following properties:\footnote{On $\T^d$, this framework works only for local well-posedness for small initial data. See Remark \ref{RMK:scaling} for a discussion of local well-posedness on periodic domains for large initial data.}

\smallskip \noi
(i) (Monotonicity) If $|\ft{f}| \leq |\ft{g}|$ pointwise, then 
\begin{align}
\| f \|_{W^s} \leq \| g \|_{W^s}.
\label{mono_gen}
\end{align}

\noi 
Here, $\ft f$ is the space-time Fourier transform of $f$.

\smallskip \noi
(ii) ($H^s$-energy estimate) The following inequality holds:
\begin{align}
\big\| \jb{\xi}^s \ft{f} (\xi, \tau) \big\|_{L_\xi^2 L_\tau^1} \les \| f \|_{W^s},
\label{Hs_gen}
\end{align}

\noi
where $\jb{\cdot} = (1 + |\cdot|^2)^{\frac 12}$.

\smallskip \noi
(iii) (Homogeneous linear estimate) There exists $b \in \R$ such that
\begin{align}
\| f \|_{W^s} \les \| f \|_{X^{s, b}},
\label{lin_gen}
\end{align}

\noi
where the $X^{s, b}$-norm is as defined in \eqref{Xsb}.

\smallskip \noi
(iv) (Bilinear estimate) The following inequality holds:
\begin{align}
\big\| \jb{\tau + |\xi|^2}^{-1} \mathcal{B} (\ft{f}, \ft{g}) \big\|_{\ft{W}^s} \les \| f \|_{W^s} \| g \|_{W^s},
\label{bilin_gen}
\end{align}

\noi
where $\ft{W}^s$ is the same norm $W^s$ on the Fourier side and $\mathcal{B} (f, g)$ is equal to $f * g$ (if $\mathcal{N} (u, u) = u^2$), $\cj{\wt f} * \cj{\wt g}$ (if $\mathcal{N} (u, u) = \cj{u}^2$), or $f * \cj{\wt g}$ (if $\mathcal{N} (u, u) = |u|^2$). Here, $\wt f (\xi, \tau) = f (- \xi, -\tau)$.

\smallskip
Now the task is to find suitable function spaces that satisfy the properties listed above. From now on, we restrict our attention to the nonlinearity $\mathcal{N} (u, u) = \cj{u}^2$ and the domain $\T^2$. As we have seen in the previous subsection, the usual $X^{s, b}$-bilinear estimate fails when the regularity is {\it very} low. This failure is caused by certain ``dangerous'' interactions. Thus, we need to introduce modifications on the $X^{s, b}$-space in order to reduce the effect by those ``dangerous'' interactions. In the following, we discuss several examples of such interactions and our strategy to deal with them. 

\begin{example} \rm
\label{EX:1}
For a large number $N \in \mathbb{N}$, let 
\begin{align*}
\ft{u_N} (n, \tau) &= \ind_{\{ n = N e_1 \}} \ind_{[-1, 1]} (\tau + N^2), \\
\ft{v_N} (n, \tau) &= \ind_{\{ n = - N e_1 \}} \ind_{[-1, 1]} (\tau + N^2),
\end{align*}

\noi
where $e_1 = (1, 0)$. Note that $\| u_N \|_{X^{s, b}} \sim N^s$ and $\| v_N \|_{X^{s, b}} \sim N^s$. A direct computation yields
\begin{align*}
\ft{\cj{u_N} \cj{v_N}}  (n, \tau) \geq \ind_{\{n = 0\}} \ind_{[-1, 1]} (- \tau + 2N^2),
\end{align*}

\noi
and so $\| \cj{u_N} \cj{v_N} \|_{X^{s, b - 1}} \ges N^{2b - 2}$. Thus, the bilinear estimate \eqref{bi2} holds only if $2b - 2 \leq 2s$ or $s \geq b - 1$. Since we need $b > \tfrac 12$, we require that $s > -\tfrac 12$.
\end{example}

In the above example, the frequency interaction is ``high-high to low'' and the modulation interaction is ``low-low to high''. However, the modulation for $\ft{\cj{u_N} \cj{v_N}}$ is not high enough for the desired $X^{s,b}$-bilinear estimate when $s \leq -\tfrac 12$. To control the above interaction when $s \leq - \tfrac 12$, we consider the following $Y^{s, b}$-norm introduced by Kishimoto \cite{Kish09}:
\begin{align*}
\| u \|_{Y^{s, b}} := \big\| \jb{n}^s \ft u (n, \tau) \big\|_{\l_n^2 L_\tau^1 (\Z^2 \times \R)} + \big\| \jb{\tau + |n|^2}^{\frac{s}{2} + b} \ft u (n, \tau) \big\|_{\l_n^2 L_\tau^2 (\Z^2 \times \R)},
\end{align*}

\noi
and we define the space $Z^{s, b} = X^{s, b} + Y^{s, b}$ via the norm
\begin{align*}
\| u \|_{Z^{s, b}} := \inf \{ \| u_1 \|_{X^{s, b}} + \| u_2 \|_{Y^{s, b}}: u_1 + u_2 = u \}.
\end{align*}

\noi
The $\l_n^2 L_\tau^1$-term in the $Y^{s, b}$-norm is needed to ensure that the $Z^{s, b}$-norm satisfies the $H^s$-energy estimate \eqref{Hs_gen}. It is not hard to check that the $Z^{s, b}$-norm satisfies the monotonicity property \eqref{mono_gen}, the $H^s$-energy estimate \eqref{Hs_gen}, and the homogeneous linear estimate \eqref{lin_gen}. Note that for $s \leq 0$ and $b > \frac 12$, if $\supp \ft u \subset \{ |\tau + |n|^2| \les |n|^2 \}$, then we have 
\begin{align*}
\| u \|_{Z^{s, b}} \sim \| u \|_{X^{s, b}} \les \| u \|_{Y^{s, b}};
\end{align*} 

\noi
if $\supp \ft u \subset \{ |\tau + |n|^2| \ges |n|^2 \}$, then we have 
\begin{align*}
\| u \|_{Z^{s, b}} \sim \| u \|_{Y^{s, b}} \les \| u \|_{X^{s, b}}.
\end{align*} 

\noi
In Section \ref{SUBSEC:Zsb}, we will revisit this $Z^{s, b}$-norm, which will be defined in a more precise manner for practical purposes.

In Example \ref{EX:1}, because of the high modulation of $\ft{\cj{u_N} \cj{v_N}}$, the $\ft{Z}^{s, b}$-norm (i.e.~the $Z^{s, b}$-norm on the Fourier side) of $\jb{\tau + |n|^2}^{-1} \ft{\cj{u_N} \cj{v_N}}$ is small enough to obtain the desired bilinear estimate \eqref{bilin_gen} for $s \leq - \frac 12$. One can easily check that using the $Z^{s, b}$-norm, the bilinear estimate for the above example holds for $s \geq 2b - 2$. This is better than $s > -\tfrac 12$ as long as $\tfrac 12 < b \leq \tfrac 34$.

Let us take a look at another example using the $Z^{s, b}$-norm assuming that $s \leq 0$.

\begin{example} \rm
\label{EX:2}
For a large number $N \in \mathbb{N}$, let 
\begin{align*}
\ft{u_N} (n, \tau) &= \ind_{\{n = N e_1\}} \ind_{[-1, 1]} (\tau + N^2), \\
\ft{v_N} (n, \tau) &= \ind_{\{n = - N e_1\}} \ind_{[-1, 1]} (-\tau + N^2).
\end{align*}

\noi
A direct computation yields
\begin{align*}
\ft{\cj{u_N} \cj{v_N}} (n, \tau) \geq \ind_{\{n = 0\}} \ind_{[-1, 1]} (\tau).
\end{align*}

\noi
Note that in this example, the frequency interaction is ``high-high to low'' and the modulation interaction is ``low-high to low''.
We can compute their corresponding $Z^{s,b}$-norms as follows: 
\begin{align*}
&\| u_N \|_{Z^{s, b}} \sim \| u_N \|_{X^{s, b}} \sim N^s \\ 
&\| v_N \|_{Z^{s, b}} \sim \| v_N \|_{Y^{s, b}} \sim N^{s + 2b} \\
&\big\| \jb{\tau + |n|^2}^{-1} \ft{\cj{u_N} \cj{v_N}} \big\|_{\ft{Z}^{s, b}} \ges 1.
\end{align*}

\noi
Thus, the bilinear estimate \eqref{bilin_gen} with $W = Z^{s, b}$ holds only if $0 \leq 2s + 2b$ or $s \geq -b$.
\end{example}

Combining Example \ref{EX:1} and Example \ref{EX:2}, we notice that the regularity $s$ needs to satisfy $s \geq 2b - 2$ and $s \geq -b$. These two lower bounds become optimal when $b = \frac 23$, so that $s = - \frac 23$ seems to be the threshold of the bilinear estimate \eqref{bilin_gen} with $W^s = Z^{s, b}$. In fact, we will show in Section \ref{SEC:bilin} that the bilinear estimate \eqref{bilin_gen} with $W^s = Z^{s, \frac 23}$ holds when $s > - \frac 23$ (see Remark \ref{RMK:besov} for a discussion on the slight loss of regularity). See Section \ref{SEC:bilin} for more details.

We conclude this introduction by stating several remarks.

\begin{remark} \rm
\label{RMK:scaling}
On $\T^d$, it is possible to use a scaling argument to prove local well-posedness for large initial data given that one can first obtain small data local well-posedness. See \cite{CKSTT04}. However, we do not pursue the scaling argument in this paper and instead rely on the time localization (Lemma \ref{LEM:time}) to prove local well-posedness for large initial data.
\end{remark}

\begin{remark} \rm
\label{RMK:besov}
In \cite{BT, Kish08}, a Besov refinement was considered in constructing desired function spaces so that the endpoint regularity (i.e.~$s = -1$ for the quadratic NLS \eqref{qNLS_gen} on $\R$ with $\mathcal{N}(u, u) = u^2$ or $\cj{u}^2$) can be handled. Similar Besov refinements were used by \cite{Guo, Kish09k} in the context of the Korteweg-de Vris equation.

For the quadratic NLS \eqref{qNLS} on $\T^2$, however, such Besov modification does not seem to be enough to cover the case when $s = - \frac 23$. This is mainly due to the fact that our approach relies heavily on the $L^4$-Strichartz estimate on $\T^2$ (see Lemma \ref{LEM:L4}), which has an $\eps$ loss of derivative.

For the quadratic NLS \eqref{qNLS} on $\T$, since the $L^4$-Strichartz estimate on $\T$ (see \cite{Bour93}) does not have any derivative loss, it seems possible to adapt the Besov modification to our estimate so that the endpoint case can be included.
\end{remark}

\begin{remark} \rm
\label{RMK:gap}
For the quadratic NLS \eqref{qNLS} on $\T$ and $\T^2$, there are still gaps between local well-posedness and ill-posedness results (see Table \ref{table2}). Specifically, on $\T$, well-posedness issues of \eqref{qNLS} for $-1 \leq s \leq - \frac 23$ remain open; on $\T^2$, well-posedness issues of \eqref{qNLS} for $-1 < s \leq - \frac 23$ remain open. One possible strategy for improving our local well-posedness arguments is to introduce weighted spaces as in \cite{BS, BT, Kish08, Kish09} in the context of Euclidean spaces.
\end{remark}

\begin{remark} \rm
\label{RMK:uu}
Let us consider the quadratic NLS \eqref{qNLS_gen} with $\mathcal{N} (u, u) = u^2$. On $\T$, local well-posedness is known to hold for $s > -\frac 12$ and ill-posedness holds for $s < -1$. We believe that the method of using modified function spaces should be able to produce better local well-posedness results, but one may need to use the weighted spaces as in \cite{BS, BT} to handle the corresponding bilinear estimate.

For the quadratic NLS \eqref{qNLS_gen} with $\mathcal{N} (u, u) = u^2$ on $\T^2$, local well-posedness is known to hold for $s \geq 0$ and ill-posedness holds for $s \leq -1$. However, it seems unlikely that the method of finding modified function spaces as illustrated at the beginning of Subsection \ref{SUBSEC:BT} works in the range $s < 0$. This is due to the following example in \cite{Grun}. For a large number $N \in \N$, let
\begin{align*}
\ft{u_N} (n, \tau) = \ind_{\{n = N e_1\}} \ind_{[-1, 1]} (\tau + N^2), \\
\ft{v_N} (n, \tau) = \ind_{\{n = N e_2\}} \ind_{[-1, 1]} (\tau + N^2),
\end{align*}

\noi
where $e_1 = (1, 0)$ and $e_2 = (0, 1)$. A direct computation yields
\begin{align*}
\ft{u_N v_N} (n, \tau) &= \ind_{\{n = N (e_1 + e_2) \}} \max \big\{ 0, \min \{ 2 - \tau - 2 N^2, 2 + \tau + 2 N^2 \} \big\} \\
&\geq \ind_{\{n = N (e_1 + e_2)\}} \ind_{[-1, 1]} (\tau + 2N^2).
\end{align*}

\noi
In this example, the frequency interaction is ``high-high to high'' and the modulation interaction is ``low-low to low'', which means that there seems to be no way to utilize the modulation to improve the bilinear estimate. Note that this ``low-low to low'' interaction does not occur for the nonlinearity $\mathcal{N} (u, u) = \cj{u}^2$, which can be seen from the computations at the beginning of Subcase 2.3 of Lemma \ref{LEM:HL} below and Case 3 of Lemma \ref{LEM:HH} below.

For any $s \in \R$ and $b \in \R$, we have
\begin{align*}
\| u_N \|_{X^{s, b}} \sim \| v_N \|_{X^{s, b}} \sim \| \jb{\tau + |n|^2}^{-1} \ft{u_N v_N} \|_{\ft{X}^{s, b}} \sim N^s,
\end{align*} 

\noi
where the $\ft{X}^{s, b}$-norm is the $X^{s, b}$-norm on the Fourier side. Thus, we observe that due to the homogeneous linear estimate \eqref{lin_gen} and the similar structures of $\ft{u_N}$, $\ft{v_N}$, and $\ft{u_N v_N}$, any qualified modified norm $\| \cdot \|_{W^s}$ should decrease the corresponding norms of $\ft{u_N}$, $\ft{v_N}$, and $\jb{\tau + |n|^2}^{-1} \ft{u_N v_N}$ with the same rate (with respect to $N$). Suppose that there exists $a \geq 0$ such that
\begin{align*}
\| u_N \|_{W^s} \sim \| v_N \|_{W^s} \sim \| \jb{\tau + |n|^2}^{-1} \ft{u_N v_N} \|_{\ft{W}^s} \sim N^{s - a}.
\end{align*}

\noi
where the $\ft{W}^s$-norm is the $W^s$-norm on the Fourier side. Then, for the bilinear estimate \eqref{bilin_gen} to hold, we must have
\begin{align*}
N^{s - a} \les N^{2s - 2a},
\end{align*}

\noi
so that $s - a \geq 0$ or $s \geq a \geq 0$. Therefore, we do not expect that the method of finding the $W^s$-norm for proving local well-posedness works for the quadratic NLS \eqref{qNLS_gen} with $\mathcal{N} (u, u) = u^2$ on $\T^2$ for $s < 0$, and it is possible that some ill-posedness results may hold in this range.
\end{remark}

\section{Notations and function spaces}
\label{SEC:func}

In this section, we introduce some notations and function spaces that enable us to prove local well-posedness of \eqref{qNLS} in low regularity settings. 

\subsection{Notations}
Throughout this paper, we drop the inessential factor of $2\pi$. For a space-time distribution $u$, we write $\ft u$ or $\F_{x, t} u$ to denote the space-time Fourier transform of $u$. If a function $\phi$ only has a space (or time) variable, then we use $\ft \phi$ to denote the Fourier transform of $\phi$ with respect to the space (or time, respectively) variable. For any function $f$, the function $\wt{f}$ is the reflection of $f$, i.e. $\wt{f}(x) = f(-x)$. We also set $\jb{\,\cdot\,} = (1 + |\cdot|^2)^\frac{1}{2}$.


We use $A \les B$ to denote $A \leq CB$ for some constant $C > 0$. We write $A \sim B$ if we have $A \les B$ and $B \les A$. We may use subscripts to denote dependence on external parameters. We also use $a +$ and $a -$ to denote $a + \eps$ and $a - \eps$, respectively, for sufficiently small $\eps > 0$.

Given a dyadic number $N \in 2^{\N \cup \{0\}}$, if $N \geq 2$, we let $P_N$ be the spatial frequency projector onto the frequencies 
\begin{align*}
\Pf_N := \{ (n, \tau) \in \Z^2 \times \R: \tfrac 12 N < |n| \leq N \}.
\end{align*}

\noi
If $N = 1$, we let $P_1$ be the spatial frequency projector onto the frequencies
\begin{align*}
\Pf_1 := \{ (n, \tau) \in \Z^2 \times \R : |n|  \leq 1 \}.
\end{align*}

\noi
For a space-time distribution $u$, we also write $u_N := P_N u$ for simplicity.

\subsection{Fourier restriction norm method}
\label{SUBSEC:Xsb}
In this subsection, we recall the definition and estimates of $X^{s, b}$-spaces for the Schr\"odinger equations, which were first introduced by Bourgain \cite{Bour93}. Given $s, b \in \R$, we define the space $X^{s, b} = X^{s, b} (\T^2 \times \R)$ to be the completion of functions that are smooth in space and Schwartz in time with respect to the following norm:
\begin{align}
\| u \|_{X^{s, b}} := \big\| \jb{n}^s \jb{\tau + |n|^2}^b \ft u (n, \tau) \big\|_{\l_n^2 L_\tau^2 (\Z^2 \times \R)}.
\label{Xsb}
\end{align}



We now present and recall some estimates related to $X^{s, b}$-norms, starting with the following stronger version of the usual homogeneous linear estimate of the $X^{s, b}$-norm as in \cite{Bour93, Tao}.
\begin{lemma}
\label{LEM:lin_tk}
Let $\varphi$ be a smooth function supported on $[-2, 2]$. Let $s \in \R$, $b \leq 1$, and $k \in \N \cup \{0\}$. Then, we have
\begin{align}
\big\| t^k \varphi(t) e^{it \Dl} \phi \big\|_{X^{s, b}} \les_{\varphi} 3^k \| \phi \|_{H^s (\T^2)}
\label{lin_tk}
\end{align}
\end{lemma}

\begin{proof}
Note that by the fact that $b \leq 1$, we have
\begin{align*}
\big\| t^k \varphi(t) e^{it \Dl} \phi \big\|_{X^{s, b}} &= \Big\| \ft{(\cdot)^k \varphi} (\tau + |n|^2) \jb{\tau + |n|^2}^b \jb{n}^s \ft \phi (n) \Big\|_{\l_n^2 L_\tau^2} \\
&= \big\| t^k \varphi (t) \big\|_{H^b (\R)} \| \phi \|_{H^s (\T^2)} \\
&\les \big( \big\| t^k \varphi (t) \big\|_{L^2 (\R)} + \big\| \dt \big( t^k \varphi (t) \big) \big\|_{L^2 (\R)} \big) \| \phi \|_{H^s (\T^2)} \\
&\les (2^k + k 2^{k - 1} ) (\| \varphi \|_{L^2 (\R)} + \| \dt \varphi \|_{L^2 (\R)}) \| \phi \|_{H^s (\T^2)} \\
&\les_\varphi 3^k \| \phi \|_{H^s (\T^2)},
\end{align*}

\noi
as desired.
\end{proof}

\begin{remark} \rm
In fact, the estimate \eqref{lin_tk} holds for all $b \in \R$. For the proof of our local well-posedness result, however, we will only need the estimate \eqref{lin_tk} for $b \leq 1$.
\end{remark}

Next, we recall the following time localization estimate. For a proof, see \cite{Bour93, Tao}.
\begin{lemma}
\label{LEM:time}
Let $s \in \R$, $-\frac 12 < b_1 \leq b_2 < \frac 12$, and $0 < T \leq 1$. Let $\varphi$ be a Schwartz function and let $\varphi_T (t) := \varphi (t / T)$. Then, we have
\begin{align*}
\| \varphi_T u \|_{X^{s, b_1}} \les_\varphi T^{b_2 - b_1} \| u \|_{X^{s, b_2}}.
\end{align*}
\end{lemma}

We also record the following $L^4$-Strichartz estimate on $\T^2$. For a proof, see \cite{Bour93, Bour95}.
\begin{lemma}
\label{LEM:L4}
Let $N$ be a dyadic number. Then, we have
\begin{align*}
\| u_N \|_{L_t^4([-1, 1]; L_x^4(\T^2))} \les N^{s} \| u_N \|_{X^{0, b}},
\end{align*}

\noi
where $0 < s < \frac 12$ and $b > \frac{1 - s}{2}$.
\end{lemma}

\subsection{Modified function spaces}
\label{SUBSEC:Zsb}
In this subsection, we define our solution space for the quadratic NLS \eqref{qNLS} in the low regularity setting and establish corresponding linear estimates. 

Given $s, b \in \R$, we define the space $Y^{s, b} = Y^{s, b} (\T^2 \times \R)$ to be the completion of functions that are smooth in space and Schwartz in time with respect to the norm
\begin{align}
\begin{split}
\| u \|_{Y^{s, b}} &:= \big\| \jb{n}^s \ft u (n, \tau) \big\|_{\l_n^2 L_\tau^1 (\Z^2 \times \R)} + \big\| \jb{\tau + |n|^2}^{\frac{s}{2} + b} \ft u (n, \tau) \big\|_{\l_n^2 L_\tau^2 (\Z^2 \times \R)}.
\end{split}
\label{Ysb}
\end{align}

\noi
The idea of this modification comes from Kishimoto \cite{Kish09}.


We now define the space $Z^{s, b}$ via the norm
\begin{align}
\| u \|_{Z^{s, b}} := \| P_{\text{lo}} u \|_{X^{s, b}} + \| P_{\text{hi}} u \|_{Y^{s, b}},
\label{Zsb}
\end{align}


\noi
where $P_{\text{lo}}$ is the space-time frequency projector onto the frequencies $\{ |\tau + |n|^2| < 2^{-10} |n|^2 \}$ and $P_{\text{hi}}$ is the space-time frequency projector onto the frequencies $\{ |\tau + |n|^2| \geq 2^{-10} |n|^2 \}$. From the definition, we observe that the $Z^{s, b}$-norm has the monotonicity property: if $|\ft{u_1}| \leq |\ft{u_2}|$ pointwise, then
\begin{align}
\| u_1 \|_{Z^{s, b}} \leq \| u_2 \|_{Z^{s, b}}.
\label{mono}
\end{align}

\noi
For $T > 0$, we define the space $Z_T^{s, b}$ as the restriction of the $Z^{s, b}$-space onto the time interval $[-T, T]$ via the norm:
\begin{align}
\| u \|_{Z_T^{s, b}} := \inf \big\{ \| v \|_{Z^{s, b}} : v |_{[-T, T]} = u \big\}.
\label{ZsbT}
\end{align}

\noi
Note that the $Z_T^{s, b}$-space is complete.

For convenience and conciseness, later on we may use the notations $\ft{X}^{s, b}$, $\ft{Y}^{s, b}$, and $\ft{Z}^{s, b}$ to denote the corresponding norms on the Fourier side. In other words, for a complex-valued function $f$ defined on $\Z^2 \times \R$, we write
\begin{align*}
\| f \|_{\ft{X}^{s, b}} &= \| \F^{-1} (f) \|_{X^{s, b}}, \\
\| f \|_{\ft{Y}^{s, b}} &= \| \F^{-1} (f) \|_{Y^{s, b}}, \\
\| f \|_{\ft{Z}^{s, b}} &= \| \F^{-1} (f) \|_{Z^{s, b}},
\end{align*}

\noi
where $\F^{-1}$ is the inverse Fourier transform.

We now establish some linear estimates of the $Z^{s, b}$-norm. We start with the following $H^s$-energy estimate.
\begin{lemma}
\label{LEM:est1}
Let $s \in \R$ and $b > \frac 12$. Then, we have
\begin{align*}
\big\| \jb{n}^s \ft u (n, \tau) \big\|_{\l_n^2 L_\tau^1} \les \| u \|_{Z^{s, b}}.
\end{align*}
\end{lemma}

\begin{proof}
By the definition of the $Z^{s, b}$-norm in \eqref{Zsb}, we know that it suffices to show the following two estimates:
\begin{align}
\big\| \jb{n}^s \ft u (n, \tau) \big\|_{\l_n^2 L_\tau^1} &\les \| u \|_{X^{s, b}}, \label{est1-1} \\
\big\| \jb{n}^s \ft u (n, \tau) \big\|_{\l_n^2 L_\tau^1} &\les \| u \|_{Y^{s, b}}. \label{est1-2}
\end{align}

\noi
Since $b > \frac 12$, we use the Cauchy-Schwarz inequality in $\tau$ to obtain
\begin{align*}
\big\| \jb{n}^s \ft u (n, \tau) \big\|_{\l_n^2 L_\tau^1} \les \big\| \jb{n}^s \jb{\tau + |n|^2}^b \ft u (n, \tau) \big\|_{\l_n^2 L_\tau^2} \leq \| u \|_{X^{s, b}},
\end{align*}

\noi
so that we obtain \eqref{est1-1}. Also, note that \eqref{est1-2} is easily obtained from the definition of the $Y^{s, b}$-norm in \eqref{Ysb}.
\end{proof}

The above lemma implies the following embedding result.

\begin{lemma}
\label{LEM:embed}
Let $s \in \R$, $b > \frac 12$, and $T > 0$. Then, we have
\begin{align*}
\| u \|_{C([-T, T]; H^s (\T^2))} \les \| u \|_{Z_T^{s, b}}.
\end{align*}

\noi
Consequently, the embedding
\begin{align*}
Z_T^{s, b} \embeds C([-T, T]; H^s (\T^2))
\end{align*}

\noi
holds.
\end{lemma}

\begin{proof}
Let $\eps > 0$ and let $v$ be an extension of $u$ outside of $[-T, T]$ such that
\begin{align}
\| v \|_{Z^{s, b}} \leq \| u \|_{Z_T^{s, b}} + \eps.
\label{vu}
\end{align}

\noi
Note that we have the following embedding 
\begin{align}
\| v \|_{C([-T, T]; H^s (\T^2))} \les \big\| \jb{n}^s \ft v (n, \tau) \big\|_{\l_n^2 L_\tau^1}.
\label{emb}
\end{align}

\noi
Thus, by \eqref{emb}, Lemma \ref{LEM:est1}, and \eqref{vu}, we obtain
\begin{align*}
\| u \|_{C([-T, T]; H^s (\T^2))} &= \| v \|_{C([-T, T]; H^s (\T^2))} \\
&\les \big\| \jb{n}^s \ft v (n, \tau) \big\|_{\l_n^2 L_\tau^1} \\
&\les \| v \|_{Z^{s, b}} \\
&\leq \| u \|_{Z_T^{s, b}} + \eps,
\end{align*}

\noi
and so the desired estimate follows since $\eps > 0$ can be arbitrarily small.
\end{proof}

Lastly, we show the following lemma, which shows that the $X^{s, b}$-space is embedded in the $Z^{s, b}$-space. 
\begin{lemma}
\label{LEM:est2}
Let $s \leq 0$ and $b > \frac 12$. Then, we have
\begin{align*}
\| u \|_{Z^{s, b}} \les \| u \|_{X^{s, b}}.
\end{align*}
\end{lemma}

\begin{proof}
We recall from \eqref{Zsb} that
\begin{align*}
\| u \|_{Z^{s, b}} = \| P_{\text{lo}} u \|_{X^{s, b}} + \| P_{\text{hi}} u \|_{Y^{s, b}},
\end{align*}

\noi
where $P_{\textup{lo}}$ projects the space-time frequencies onto $\{ |\tau + |n|^2| < 2^{-10} |n|^2 \}$ and $P_{\textup{hi}}$ projects the space-time frequencies onto $\{ |\tau + |n|^2| \geq 2^{-10} |n|^2 \}$. Note that we have
\begin{align*}
\| P_{\text{lo}} u \|_{X^{s, b}} \leq \| u \|_{X^{s, b}}.
\end{align*}

\noi
For the $\| P_{\text{hi}} u \|_{Y^{s, b}}$ term, note that by the Cauchy-Schwarz inequality, we have
\begin{align*}
\big\| \jb{n}^s \ft u (n, \tau) \big\|_{\l_n^2 L_\tau^1} \les \big\| \jb{n}^s \jb{\tau + |n|^2}^{b} \ft u (n, \tau) \big\|_{\l_n^2 L_\tau^2} = \| u \|_{X^{s, b}},
\end{align*}

\noi
since $b > \frac 12$. Also, we have
\begin{align*}
\big\| \jb{\tau + |n|^2}^{\frac{s}{2} + b} \ft u (n, \tau) \ind_{\{|\tau + |n|^2| \geq 2^{-10} |n|^2\}} \big\|_{\l_n^2 L_\tau^2} \les \big\| \jb{n}^s \jb{\tau + |n|^2}^{b} \ft u (n, \tau) \big\|_{\l_n^2 L_\tau^2} = \| u \|_{X^{s, b}}.
\end{align*}

\noi
Thus, we obtain that $\| P_{\text{hi}} u \|_{Y^{s, b}} \les \| u \|_{X^{s, b}}$, so that we achieve the desired inequality.
\end{proof}

\section{Bilinear estimate}
\label{SEC:bilin}
In this section, we establish the crucial bilinear estimate with respect to the $Z^{s, b}$-norm introduced in the previous section. Specifically, we show the following proposition.
\begin{proposition}
\label{PROP:bilin}
Let $- \frac 23 < s \leq - \frac 12$ and $0 < T \leq \frac 12$. Let $\varphi : \R \to [0, 1]$ be a smooth function such that $\varphi \equiv 1$ on $[-1, 1]$ and $\varphi \equiv 0$ outside of $[-2, 2]$, and let $\varphi_T (t) := \varphi (t / T)$. Then, we have
\begin{align*}
\big\|  \jb{\tau + |n|^2}^{-1} \F_{x, t} \big( \varphi_T \cj{u} \cdot  \varphi_T \cj{v} \big) (n, \tau)  \big\|_{\ft{Z}^{s, \frac 23}} \les_\varphi T^\ta \| u \|_{Z^{s, \frac 23}} \| v \|_{Z^{s, \frac 23}}
\end{align*}

\noi
for some $\ta > 0$.
\end{proposition}

Let us first consider two particular cases of Proposition \ref{PROP:bilin}. We start with the following ``high-low interaction'' estimate.
\begin{lemma}
\label{LEM:HL}
Let $-\frac 23 < s \leq -\frac 12$ and $0 < T \leq \frac 12$. Let $N$, $N_1$, and $N_2$ be dyadic numbers. Let $\varphi : \R \to [0, 1]$ be a smooth function such that $\varphi \equiv 1$ on $[-1, 1]$ and $\varphi \equiv 0$ outside of $[-2, 2]$, and let $\varphi_T (t) := \varphi (t / T)$.

\smallskip \noi
\textup{(i)} If $2^{-5} N \leq N_1 \leq 2^5 N$ and $N_2 \leq 2^6 N$, we have
\begin{align*}
\big\| \jb{\tau + |n|^2}^{-1} \F_{x, t} \big( \varphi_T \cj{u_{N_1}} \cdot \varphi_T \cj{v_{N_2}} \big) (n, \tau) \big\|_{\ft{Z}^{s, \frac 23} (\Pf_N)} \les_\varphi N_2^{-\dl} T^\ta \| u_{N_1} \|_{Z^{s, \frac 23}} \| v_{N_2} \|_{Z^{s, \frac 23}}
\end{align*}

\noi
for some $\dl > 0$ and $\ta > 0$.

\smallskip \noi
\textup{(ii)} If $2^{-5} N \leq N_2 \leq 2^5 N$ and $N_1 \leq 2^6 N$, we have
\begin{align*}
\big\| \jb{\tau + |n|^2}^{-1} \F_{x, t} \big( \varphi_T \cj{u_{N_1}} \cdot \varphi_T \cj{v_{N_2}} \big) (n, \tau) \big\|_{\ft{Z}^{s, \frac 23} (\Pf_N)} \les_\varphi N_1^{-\dl} T^\ta \| u_{N_1} \|_{Z^{s, \frac 23}} \| v_{N_2} \|_{Z^{s, \frac 23}}
\end{align*}

\noi
for some $\dl > 0$ and $\ta > 0$.
\end{lemma}

\begin{proof}
By the symmetry of $u$ and $v$, it suffices to prove (i). Below we use $(n_1, \tau_1)$ as the variables of $\ft{ \varphi_T u_{N_1} }$ or $\ft{u_{N_1}}$ and $(n_2, \tau_2)$ as the variables of $\ft{ \varphi_T v_{N_2} }$ or $\ft{v_{N_2}}$. Note that we have the relations $\tau + \tau_1 + \tau_2 = 0$ and $n + n_1 + n_2 = 0$. We also recall the notation $\wt{f} (x) = f(-x)$.

We divide the argument into two main cases depending on the relationship between the modulation function $\tau + |n|^2$ and $|n|^2$.

\medskip \noi
\textbf{Case 1:} $| \tau + |n|^2 | \geq 2^{-10} |n|^2$.

In this case, we need to evaluate the $\F_{x, t} \big( \varphi_T \cj{u_{N_1}} \cdot \varphi_T \cj{v_{N_2}} \big)$ term using the $\ft{Y}^{s, \frac 23}$-norm, and we need to evaluate both the $\l_n^2 L_\tau^1$ term and the $\l_n^2 L_\tau^2$ term. We consider the following three subcases.

\medskip \noi
\textbf{Subcase 1.1:} $|\tau_1 + |n_1|^2| \geq 2^{-10} |n_1|^2$.

In this subcase, we need to estimate $u_{N_1}$ using the $Y^{s, \frac 23}$-norm. By Young's convolution inequality, Lemma \ref{LEM:time}, the Cauchy-Schwarz inequality, and Lemma \ref{LEM:est1}, we obtain
\begin{align}
\begin{split}
\big\| &\jb{\tau + |n|^2}^{\frac{s}{2} - \frac 13} \F_{x, t} \big( \varphi_T \cj{u_{N_1}} \cdot \varphi_T \cj{v_{N_2}} \big) \big\|_{\l_n^2 L_\tau^2 (\Pf_N)} \\
&= \bigg\| \jb{\tau + |n|^2}^{\frac{s}{2} - \frac 13} \wt{\ft{\varphi_T^2 u_{N_1}}} * \wt{\ft{v_{N_2}}} \bigg\|_{\l_n^2 L_\tau^2 (\Pf_N)} \\
&\les N^{s - \frac 23} \Big\| \ft{\varphi_T^2 u_{N_1}} \Big\|_{\l_{n_1}^2 L_{\tau_1}^2} \big\| \ft{v_{N_2}} \big\|_{\l_{n_2}^1 L_{\tau_2}^1} \\
&\les_\varphi N^{s - \frac 23} T^{\eps} \big\| \jb{\tau_1 + |n_1|^2}^\eps \ft{u_{N_1}} \big\|_{\l_{n_1}^2 L_{\tau_1}^2} N_2^{- s + 1} \big\| \jb{n_2}^s \ft{v_{N_2}} \big\|_{\l_{n_2}^2 L_{\tau_2}^1} \\
&\les N^{s - \frac 23} T^\eps N_1^{-s - \frac 43 + 2\eps}  \| u_{N_1} \|_{Y^{s, \frac 23}} N_2^{- s + 1} \| v_{N_2} \|_{Z^{s, \frac 23}} \\
&\les N^{-2 + 2\eps} N_2^{-s + 1} T^\eps \| u_{N_1} \|_{Z^{s, \frac 23}} \| v_{N_2} \|_{Z^{s, \frac 23}},
\end{split}
\label{HL1-1}
\end{align}

\noi
where $\eps > 0$ is arbitrarily small. Since $ - s + 1 > 0$ given $s \leq -\frac 12$, the above estimate is acceptable if $-s - 1 + 2\eps < 0$, which is valid given $s > - \frac 23$ and $\eps > 0$ sufficiently small.

Also, by the Cauchy-Schwarz inequality, we get
\begin{align*}
\big\| &\jb{n}^s \jb{\tau + |n|^2}^{- 1} \F_{x, t} \big( \varphi_T \cj{u_{N_1}} \cdot \varphi_T \cj{v_{N_2}} \big) \big\|_{\l_n^2 L_\tau^1 (\Pf_N)} \\
&\les \big\| \jb{n}^s \jb{\tau + |n|^2}^{- \frac 13} \F_{x, t} \big( \varphi_T \cj{u_{N_1}} \cdot \varphi_T \cj{v_{N_2}} \big) \big\|_{\l_n^2 L_\tau^2 (\Pf_N)},
\end{align*}

\noi
which can be estimated similarly as in \eqref{HL1-1}. Combining the above two estimates, we obtain the desired inequality.

\medskip \noi
\textbf{Subcase 1.2:} $|\tau_2 + |n_2|^2| \geq 2^{-10} |n_2|^2$.

In this subcase, we need to estimate $v_{N_2}$ using the $Y^{s, \frac 23}$-norm. By Young's convolution inequality, the Cauchy-Schwarz inequality, Lemma \ref{LEM:time}, and Lemma \ref{LEM:est1}, we obtain
\begin{align}
\begin{split}
\big\| &\jb{\tau + |n|^2}^{\frac{s}{2} - \frac 13} \F_{x, t} \big( \varphi_T \cj{u_{N_1}} \cdot \varphi_T \cj{v_{N_2}} \big) \big\|_{\l_n^2 L_\tau^2 (\Pf_N)} \\
&= \bigg\| \jb{\tau + |n|^2}^{\frac{s}{2} - \frac 13} \wt{\ft{u_{N_1}}} * \wt{\ft{\varphi_T^2 v_{N_2}}} \bigg\|_{\l_n^2 L_\tau^2 (\Pf_N)} \\
&\les N^{s - \frac 23} \big\| \ft{u_{N_1}} \big\|_{\l_{n_1}^2 L_{\tau_1}^1} \Big\| \ft{\varphi_T^2 v_{N_2}} \Big\|_{\l_{n_2}^1 L_{\tau_2}^2} \\
&\les_\varphi N^{s - \frac 23} N_1^{-s} \big\| \jb{n_1}^s \ft{u_{N_1}} \big\|_{\l_{n_1}^2 L_{\tau_1}^1}  T^\eps N_2 \big\| \jb{\tau_2 + |n_2|^2}^\eps \ft{v_{N_2}} \big\|_{\l_{n_2}^2 L_{\tau_2}^2} \\
&\les N^{s - \frac 23} N_1^{-s} \| u_{N_1} \|_{Z^{s, \frac 23}} T^\eps N_2^{- s - \frac 13 + 2\eps} \| v_{N_2} \|_{Y^{s, \frac 23}} \\
&\les N^{- \frac 23} N_2^{- s - \frac 13 + 2\eps} T^\eps \| u_{N_1} \|_{Z^{s, \frac 23}} \| v_{N_2} \|_{Z^{s, \frac 23}},
\end{split}
\label{HL1-2}
\end{align}

\noi
where $\eps > 0$ is arbitrarily small. Since $-s - \frac 13 + 2\eps > 0$ given $s \leq - \frac 12$, the above estimate is acceptable if $-s - 1 + 2\eps < 0$, which is valid given $s > - \frac 23$ and $\eps > 0$ sufficiently small. 

Also, by the Cauchy-Schwarz inequality, we get
\begin{align*}
\big\| &\jb{n}^s \jb{\tau + |n|^2}^{- 1} \F_{x, t} \big( \varphi_T \cj{u_{N_1}} \cdot \varphi_T \cj{v_{N_2}} \big) \big\|_{\l_n^2 L_\tau^1 (\Pf_N)} \\
&\les \big\| \jb{n}^s \jb{\tau + |n|^2}^{- \frac 13} \F_{x, t} \big( \varphi_T \cj{u_{N_1}} \cdot \varphi_T \cj{v_{N_2}} \big) \big\|_{\l_n^2 L_\tau^2 (\Pf_N)},
\end{align*}

\noi
which can be estimated similarly as in \eqref{HL1-2}. Combining the above two estimates, we obtain the desired inequality.

\medskip \noi
\textbf{Subcase 1.3:} $|\tau_1 + |n_1|^2| < 2^{-10} |n_1|^2$ and $|\tau_2 + |n_2|^2| < 2^{-10} |n_2|^2$.

In this subcase, we need to estimate both $u_{N_1}$ and $v_{N_2}$ using the $X^{s, \frac 23}$-norm. Using the fact that $\varphi_T$ is supported on $[-1, 1]$ given $0 < T \leq \frac 12$, by the Plancherel theorem, H\"older's inequality, Lemma \ref{LEM:L4}, and Lemma \ref{LEM:time}, we obtain
\begin{align}
\begin{split}
\big\| &\jb{\tau + |n|^2}^{\frac{s}{2} - \frac 13} \F_{x, t} \big( \varphi_T \cj{u_{N_1}} \cdot \varphi_T \cj{v_{N_2}} \big) \big\|_{\l_n^2 L_\tau^2 (\Pf_N)} \\
&\les N^{s - \frac 23} \| \varphi_T \cj{u_{N_1}} \cdot \varphi_T \cj{v_{N_2}} \|_{L^2_t ([-1, 1]; L^2_x (\T^2))} \\
&\les N^{s - \frac 23} \| \varphi_T u_{N_1} \|_{L^4_t ([-1, 1]; L^4_x (\T^2))} \| \varphi_T v_{N_2} \|_{L^4_t ([-1, 1]; L^4_x (\T^2))} \\
&\les N^{s - \frac 23} N_1^{4\eps} \| \varphi_T u_{N_1} \|_{X^{0, \frac 12 - \eps}} N_2^{4\eps} \| \varphi_T v_{N_2} \|_{X^{0, \frac 12 - \eps}} \\
&\les_\varphi N^{s - \frac 23} N_1^{-s + 4\eps} T^{\frac{\eps}{2}} \| u_{N_1} \|_{X^{s, \frac 12 - \frac{\eps}{2}}} N_2^{-s + 4\eps} T^{\frac{\eps}{2}} \| v_{N_2} \|_{X^{s, \frac 12 - \frac{\eps}{2}}} \\
&\les N^{- \frac 23 + 4\eps} N_2^{-s + 4\eps} T^{\eps} \| u_{N_1} \|_{Z^{s, \frac 23}} \| v_{N_2} \|_{Z^{s, \frac 23}},
\end{split}
\label{HL1-3}
\end{align}

\noi
where $\eps > 0$ is arbitrarily small. Since $s \leq - \frac 12 < 0$, the above estimate is acceptable if $-s - \frac 23 + 8\eps < 0$, which is valid given $s > - \frac 23$ and $\eps > 0$ small enough.

Also, by the Cauchy-Schwarz inequality, we get
\begin{align*}
\big\| &\jb{n}^s \jb{\tau + |n|^2}^{- 1} \F_{x, t} \big( \varphi_T \cj{u_{N_1}} \cdot \varphi_T \cj{v_{N_2}} \big) \big\|_{\l_n^2 L_\tau^1 (\Pf_N)} \\
&\les \big\| \jb{n}^s \jb{\tau + |n|^2}^{- \frac 13} \F_{x, t} \big( \varphi_T \cj{u_{N_1}} \cdot \varphi_T \cj{v_{N_2}} \big) \big\|_{\l_n^2 L_\tau^2 (\Pf_N)},
\end{align*}

\noi
which can be estimated similarly as in \eqref{HL1-3}. Combining the above two estimates, we obtain the desired inequality.

\medskip \noi
\textbf{Case 2:} $|\tau + |n|^2| < 2^{-10} |n|^2$.

In this case, we need to evaluate the $\F_{x, t} \big( \varphi_T \cj{u_{N_1}} \cdot \varphi_T \cj{v_{N_2}} \big)$ term using the $\ft{X}^{s, \frac 23}$-norm. 

We assume that $n \neq 0$. Note that if $n = 0$, we have $N = 1$ which then implies that $N_1 \leq 2^5$ and $N_2 \leq 2^6$, and so the estimate will follow in a similar (and much easier) manner.

We consider the following three subcases.

\medskip \noi
\textbf{Subcase 2.1:} $|\tau_1 + |n_1|^2| \geq 2^{-10} |n_1|^2$ and $|\tau_2 + |n_2|^2| \geq 2^{-10} |n_2|^2$.

In this subcase, we need to estimate both $u_{N_1}$ and $v_{N_2}$ using the $Y^{s, \frac 23}$-norm. By H\"older's inequality, Young's convolution inequality, and Lemma \ref{LEM:time}, we have
\begin{align*}
\big\| &\jb{n}^s \jb{\tau + |n|^2}^{- \frac 13} \F_{x, t} \big( \varphi_T \cj{u_{N_1}} \cdot \varphi_T \cj{v_{N_2}} \big) \big\|_{\l_n^2 L_\tau^2 (\Pf_N)} \\
&= \bigg\| \jb{n}^s \jb{\tau + |n|^2}^{- \frac 13} \wt{\ft{u_{N_1}}} * \wt{\ft{\varphi_T^2 v_{N_2}}} \bigg\|_{\l_n^2 L_\tau^2 (\Pf_N)} \\
&\les \big\| \jb{n}^s \jb{\tau + |n|^2}^{- \frac 13} \big\|_{\l_n^2 L_\tau^2 (\Pf_N)} \big\| \ft{u_{N_1}} \big\|_{\l_{n_1}^2 L_{\tau_1}^2} \Big\| \ft{\varphi_T^2 v_{N_2}} \Big\|_{\l_{n_2}^2 L_{\tau_2}^2} \\
&\les_\varphi N^{s + 1} N^{\frac 13} N_1^{- s - \frac 43} \| u_{N_1} \|_{Y^{s, \frac 23}} T^\eps \big\| \jb{\tau_2 + |n_2|^2}^\eps \ft{v_{N_2}} \big\|_{\l_{n_2}^2 L_{\tau_2}^2} \\
&\les T^{\eps} \| u_{N_1} \|_{Z^{s, \frac 23}}  N_2^{-s - \frac 43 + 2 \eps}  \| v_{N_2} \|_{Y^{s, \frac 23}}  \\
&\les N_2^{-s - \frac 43 + 2\eps} T^{\eps} \| u_{N_1} \|_{Z^{s, \frac 23}} \| v_{N_2} \|_{Z^{s, \frac 23}},
\end{align*}

\noi
where $\eps > 0$ is arbitrarily small. The above estimate is acceptable if $-s - \frac 43 + 2\eps < 0$, which is valid given $s > - \frac 23$ and $\eps > 0$ sufficiently small.

\medskip \noi
\textbf{Subcase 2.2:} $| \tau_1 + |n_1|^2 | \geq 2^{-10} |n_1|^2$ and $| \tau_2 + |n_2|^2 | < 2^{-10} |n_2|^2$.

In this subcase, we need to estimate $u_{N_1}$ using the $Y^{s, \frac 23}$-norm and estimate $v_{N_2}$ using the $X^{s, \frac 23}$-norm. By duality and the Cauchy-Schwarz inequality, we have
\begin{align}
\begin{split}
\big\| &\jb{n}^s \jb{\tau + |n|^2}^{- \frac 13} \F_{x, t} \big( \varphi_T \cj{u_{N_1}} \cdot \varphi_T \cj{v_{N_2}} \big) \big\|_{\l_n^2 L_\tau^2 (\Pf_N)} \\
&\les N^s \sup_{\| h \|_{\l_n^2 L_\tau^2 (\Pf_N)} \leq 1} \bigg| \sum_{\substack{n, n_1, n_2 \in \Z^2 \\ n + n_1 + n_2 = 0}} \iint_{\tau + \tau_1 + \tau_2 = 0} \ft{\varphi_T u_{N_1}} (n_1, \tau_1) \ft{\varphi_T v_{N_2}} (n_2, \tau_2) \\
&\qquad \times \frac{h(n, \tau)}{\jb{\tau + |n|^2}^{\frac 13}} d\tau d\tau_1 \bigg| \\
&\leq N^s \big\| \ft{\varphi_T u_{N_1}} \big\|_{\l_{n_1}^2 L_{\tau_1}^2} \sup_{\| h \|_{\l_n^2 L_\tau^2 (\Pf_N)} \leq 1} \bigg\| \sum_{\substack{n, n_2 \in \Z^2 \\ n + n_1 + n_2 = 0}} \int_{\tau + \tau_1 + \tau_2 = 0} \ft{\varphi_T v_{N_2}} (n_2, \tau_2) \\
&\qquad \times \frac{h(n, \tau)}{\jb{\tau + |n|^2}^{\frac 13}} d\tau  \bigg\|_{\l_{n_1}^2 L_{\tau_1}^2}.
\end{split}
\label{HL2-2}
\end{align}

\noi
Let $w_N$ be a space-time distribution that satisfy  $\ft{w_N} (n, \tau) = h(n, \tau) / \jb{\tau + |n|^2}^{\frac 13}$. Then, using the fact that $\varphi_T$ is supported on $[-1, 1]$ given $0 < T \leq \frac 12$, by the Plancherel theorem, H\"older's inequality, Lemma \ref{LEM:L4}, and Lemma \ref{LEM:time}, we have
\begin{align*}
\bigg\| &\sum_{\substack{n, n_2 \in \Z^2 \\ n + n_1 + n_2 = 0}} \int_{\tau + \tau_1 + \tau_2 = 0} \ft{\varphi_T v_{N_2}} (n_2, \tau_2) \frac{h(n, \tau)}{\jb{\tau + |n|^2}^{\frac 13}} d\tau  \bigg\|_{\l_{n_1}^2 L_{\tau_1}^2} \\ 
&= \| \varphi_T v_{N_2} \wt{w_N} \|_{L^2_t ([-1, 1]; L^2_x (\T^2))} \\
&\les \| \varphi_T v_{N_2} \|_{L^4_t ([-1, 1]; L^4_x (\T^2))} \| w_N \|_{L^4_t ([-1, 1]; L^4_x (\T^2))} \\
&\les N_2^{4 \eps} \| \varphi_T v_{N_2} \|_{X^{0, \frac 12 - \eps}} N^{\frac 13 + \eps} \| w_N \|_{X^{0, \frac 13}} \\
&\les_\varphi N_2^{- s + 4\eps} T^{\frac{\eps}{2}} \| v_{N_2} \|_{X^{s, \frac 23}} N^{\frac 13 + \eps} \| h \|_{\l_n^2 L_\tau^2 (\Pf_N)},
\end{align*}

\noi
where $\eps > 0$ is arbitrarily small. Thus, continuing with \eqref{HL2-2}, we use Lemma \ref{LEM:time} to obtain
\begin{align*}
\big\| &\jb{n}^s \jb{\tau + |n|^2}^{- \frac 13} \F_{x, t} \big( \varphi_T \cj{u_{N_1}} \cdot \varphi_T \cj{v_{N_2}} \big) \big\|_{\l_n^2 L_\tau^2 (\Pf_N)} \\
&\les_\varphi N^{s + \frac 13 + \eps} \big\| \ft{\varphi_T u_{N_1}} \big\|_{\l_{n_1}^2 L_{\tau_1}^2} N_2^{-s + 4\eps} T^{\frac{\eps}{2}} \| v_{N_2} \|_{X^{s, \frac 23}} \\
&\les_\varphi N^{s + \frac 13 + \eps} N_2^{-s + 4\eps} T^{\eps} \big\| \jb{\tau_1 + |n_1|^2}^\frac{\eps}{2} \ft{u_{N_1}} \big\|_{\l_{n_1}^2 L_{\tau_1}^2} \| v_{N_2} \|_{Z^{s, \frac 23}} \\
&\les N^{s + \frac 13 + \eps} N_2^{-s + 4\eps} T^{\eps} N_1^{-s - \frac 43 + \eps} \| u_{N_1} \|_{Y^{s, \frac 23}} \| v_{N_2} \|_{Z^{s, \frac 23}} \\
&\les N^{- 1 + 2 \eps} N_2^{-s + 4\eps} T^{\eps} \| u_{N_1} \|_{Z^{s, \frac 23}} \| v_{N_2} \|_{Z^{s, \frac 23}}.
\end{align*}

\noi
Since $s < 0$, the above estimate is acceptable if $-s - 1 + 6 \eps < 0$, which is valid given $s > - \frac 23$ and $\eps > 0$ small enough.

\medskip \noi
\textbf{Subcase 2.3:} $|\tau_1 + |n_1|^2| < 2^{-10} |n_1|^2$.

In this subcase, we first note that
\begin{align*}
\tau < 2^{-10} |n|^2 - |n|^2 \quad \text{and} \quad \tau_1 < 2^{-10} |n_1|^2 - |n_1|^2.
\end{align*}

\noi
Note that since we assumed $n \neq 0$, we have
\begin{align*}
\tau_2 = - \tau - \tau_1 > |n|^2 - 2^{-10} |n|^2 + |n_1|^2 - 2^{-10} |n_1|^2 > \frac{1}{2} |n|^2.
\end{align*}

\noi
Thus, we have 
\begin{align}
|\tau_2 + |n_2|^2| \ges N^2
\label{addi}
\end{align}
and $|\tau_2 + |n_2|^2| \geq |n_2|^2 > 2^{-10} |n_2|^2$.

We need to estimate $u_{N_1}$ using the $X^{s, \frac 23}$-norm and estimate $v_{N_2}$ using the $Y^{s, \frac 23}$-norm. By using similar steps as in Subcase 2.2 by switching the roles of $u_{N_1}$ and $v_{N_2}$ along with the additional condition \eqref{addi}, we obtain
\begin{align*}
\big\| &\jb{n}^s \jb{\tau + |n|^2}^{- \frac 13} \F_{x, t} \big( \varphi_T \cj{u_{N_1}} \cdot \varphi_T \cj{v_{N_2}} \big) \big\|_{\l_n^2 L_\tau^2 (\Pf_N)}  \\
&\les_\varphi N^{s + \frac 13 + \eps} N_1^{-s + 4\eps} T^{\frac{\eps}{2}} \| u_{N_1} \|_{X^{s, \frac 23}} \big\| \ft{\varphi_T v_{N_2}} \big\|_{\l_{n_2}^2 L_{\tau_2}^2} \\
&\les_\varphi N^{\frac 13 + 5 \eps} T^\eps \| u_{N_1} \|_{Z^{s, \frac 23}} \big\| \jb{\tau_2 + |n_2|^2}^{\frac{\eps}{2}} \ft{v_{N_2}} \big\|_{\l_{n_2}^2 L_{\tau_2}^2} \\
&\les N^{\frac 13 + 5 \eps} T^\eps \| u_{N_1} \|_{Z^{s, \frac 23}} N^{-s - \frac 43 + \eps} \| v_{N_2} \|_{Y^{s, \frac 23}} \\
&\les N^{-s - 1 + 6 \eps} T^\eps \| u_{N_1} \|_{Z^{s, \frac 23}} \| v_{N_2} \|_{Z^{s, \frac 23}}.
\end{align*}

\noi
where $\eps > 0$ is arbitrarily small. The above estimate is acceptable if $-s - 1 + 6 \eps < 0$, which is valid given $s > - \frac 23$ and $\eps > 0$ small enough.

\medskip
Thus, we have finished our proof.
\end{proof}

We now show the following ``high-high interaction'' estimate.
\begin{lemma}
\label{LEM:HH}
Let $-\frac 23 < s \leq -\frac 12$ and $0 < T \leq \frac 12$. Let $N$, $N_1$, and $N_2$ be dyadic numbers such that $\frac 12 N_1 \leq N_2 \leq 2 N_1$ and $N < 2^{-5} N_1$. Let $\varphi: \R \to [0, 1]$ be a smooth function such that $\varphi \equiv 1$ on $[-1, 1]$ and $\varphi \equiv 0$ outside of $[-2, 2]$, and let $\varphi_T (t) := \varphi (t / T)$. Then, we have
\begin{align*}
\big\| \jb{\tau + |n|^2}^{-1} \F_{x, t} \big( \varphi_T \cj{u_{N_1}} \cdot \varphi_T \cj{v_{N_2}} \big) (n, \tau) \big\|_{\ft{Z}^{s, \frac 23} (\Pf_N)} \les_\varphi N^{-\dl} T^\ta \| u_{N_1} \|_{Z^{s, \frac 23}} \| v_{N_2} \|_{Z^{s, \frac 23}}
\end{align*}

\noi
for some $\dl > 0$ and $\ta > 0$.
\end{lemma}

\begin{proof}
As in the proof of the previous lemma, we use $(n_1, \tau_1)$ as the variables of $\cj{\varphi_T u_{N_1}}$ or $\cj{u_{N_1}}$, and $(n_2, \tau_2)$ as the variables of $\cj{\varphi_T v_{N_2}}$ or $\cj{v_{N_2}}$. Note that we have the relations $\tau + \tau_1 + \tau_2 = 0$ and $n + n_1 + n_2 = 0$. Also, the assumptions on the sizes of $N$, $N_1$, and $N_2$ ensure that $n_1 \neq 0$ and $n_2 \neq 0$. We also recall the notation $\wt{f} (x) = f(-x)$.

We consider the following four main cases.

\medskip \noi
\textbf{Case 1:} $|\tau + |n|^2| \geq 2^{-10} |n_1|^2$.

In this case, we have $|\tau + |n|^2| \geq 2^{-10} |n_1|^2 \geq 2^{-10} |n|^2$ given $N < 2^{-5} N_1$, so that we need to evaluate the $\F_{x, t} \big( \varphi_T \cj{u_{N_1}} \cdot \varphi_T \cj{v_{N_2}} \big)$ term using the $\ft{Y}^{s, \frac 23}$-norm, and we need to evaluate both the $\l_n^2 L_\tau^1$ term and the $\l_n^2 L_\tau^2$ term. We consider the following three subcases.

\medskip \noi
\textbf{Subcase 1.1:} $|\tau_1 + |n_1|^2| \geq 2^{-10} |n_1|^2$.

In this subcase, we need to estimate $u_{N_1}$ using the $Y^{s, \frac 23}$-norm. By Young's convolution inequality, Lemma \ref{LEM:time}, the Cauchy-Schwarz inequality, and Lemma \ref{LEM:est1}, we obtain
\begin{align*}
\big\| &\jb{\tau + |n|^2}^{\frac{s}{2} - \frac 13} \F_{x, t} \big( \varphi_T \cj{u_{N_1}} \cdot \varphi_T \cj{v_{N_2}} \big) \big\|_{\l_n^2 L_\tau^2 (\Pf_N)} \\
&= \bigg\| \jb{\tau + |n|^2}^{\frac{s}{2} - \frac 13} \wt{\ft{\varphi_T^2 u_{N_1}}} * \wt{\ft{v_{N_2}}} \bigg\|_{\l_n^2 L_\tau^2 (\Pf_N)} \\
&\les N_1^{s - \frac 23} \Big\| \ft{\varphi_T^2 u_{N_1}} \Big\|_{\l_{n_1}^2 L_{\tau_1}^2} \big\| \ft{v_{N_2}} \big\|_{\l_{n_2}^1 L_{\tau_2}^1} \\
&\les_\varphi N_1^{s - \frac 23} T^\eps \big\| \jb{\tau_1 + |n_1|^2}^\eps \ft{u_{N_1}} \big\|_{\l_{n_1}^2 L_{\tau_1}^2} N_2^{-s + 1} \big\| \jb{n_2}^s \ft{v_{N_2}} \big\|_{\l_{n_2}^2 L_{\tau_2}^1} \\
&\les N_1^{s - \frac 23} T^\eps N_1^{-s - \frac 43 + 2\eps} \| u_{N_1} \|_{Y^{s, \frac 23}} N_2^{-s + 1} \| v_{N_2} \|_{Z^{s, \frac 23}} \\
&\les N_1^{-s - 1 + 2 \eps} T^\eps \| u_{N_1} \|_{Z^{s, \frac 23}} \| v_{N_2} \|_{Z^{s, \frac 23}},
\end{align*}

\noi
which is acceptable given $s > - \frac 23$ and $\eps > 0$ sufficiently small.

Also, by H\"older's inequality, Young's convolution inequality, Lemma \ref{LEM:time}, and Lemma \ref{LEM:est1}, we have
\begin{align*}
\big\| &\jb{n}^s \jb{\tau + |n|^2}^{-1} \F_{x, t} \big( \varphi_T \cj{u_{N_1}} \cdot \varphi_T \cj{v_{N_2}} \big) \big\|_{\l_n^2 L_\tau^1 (\Pf_N)} \\
&= \bigg\| \jb{n}^s \jb{\tau + |n|^2}^{-1} \wt{\ft{\varphi_T^2 u_{N_1}}} * \wt{\ft{v_{N_2}}} \bigg\|_{\l_n^2 L_\tau^1 (\Pf_N)} \\
&\les N_1^{- 1 + 2 \eps} \big\| \jb{n}^s \jb{\tau + |n|^2}^{-\frac 12 - \eps} \big\|_{\l_n^2 L_\tau^2 (\Pf_N)} \Big\| \ft{\varphi_T^2 u_{N_1}} \Big\|_{\l_{n_1}^2 L_{\tau_1}^2} \big\| \ft{v_{N_2}} \big\|_{\l_{n_2}^2 L_{\tau_2}^1} \\
&\les_\varphi N_1^{- 1 + 2 \eps} N^{s + 1} T^\eps \big\| \jb{\tau_1 + |n_1|^2}^\eps \ft{u_{N_1}} \big\|_{\l_{n_2}^2 L_{\tau_2}^2} N_2^{-s} \big\| \jb{n_2}^s \ft{v_{N_2}} \big\|_{\l_{n_2}^2 L_{\tau_2}^1} \\
&\les N^{s + 1} N_1^{-s - 1 + 2\eps} T^\eps N_1^{-s - \frac 43 + 2 \eps} \| u_{N_1} \|_{Y^{s, \frac 23}} \| v_{N_2} \|_{Z^{s, \frac 23}} \\
&\les N^{s + 1} N_1^{-2s - \frac 73 + 4 \eps} T^\eps \| u_{N_1} \|_{Z^{s, \frac 23}} \| v_{N_2} \|_{Z^{s, \frac 23}},
\end{align*}

\noi
where $\eps > 0$ is arbitrarily small. Since $s + 1 > 0$ given $s > - \frac 23$, the above estimate is acceptable if $-s - \frac 43 + 4 \eps < 0$, which is valid given $s > -\frac 23$ and $\eps > 0$ small enough. Combining the above two estimates, we obtain the desired inequality.

\medskip \noi
\textbf{Subcase 1.2:} $|\tau_2 + |n_2|^2| \geq 2^{-10} |n_2|^2$.

This subcase is similar to Subcase 1.1 by switching the roles of $u_{N_1}$ and $v_{N_2}$, and so we omit details.

\medskip \noi
\textbf{Subcase 1.3:} $|\tau_1 + |n_1|^2| < 2^{-10} |n_1|^2$ and $|\tau_2 + |n_2|^2| < 2^{-10} |n_2|^2$.

In this subcase, we need to estimate both $u_{N_1}$ and $v_{N_2}$ using the $X^{s, \frac 23}$-norm. Using the fact that $\varphi_T$ is supported on $[-1, 1]$ given $0 < T \leq \frac 12$, by the Plancherel theorem, H\"older's inequality, Lemma \ref{LEM:L4}, and Lemma \ref{LEM:time}, we obtain
\begin{align*}
\big\| &\jb{\tau + |n|^2}^{\frac{s}{2} - \frac 13} \F_{x, t} \big( \varphi_T \cj{u_{N_1}} \cdot \varphi_T \cj{v_{N_2}} \big) \big\|_{\l_n^2 L_\tau^2 (\Pf_N)} \\
&\les N_1^{s - \frac 23} \| \varphi_T \cj{u_{N_1}} \cdot \varphi_T \cj{v_{N_2}} \|_{L^2_t ([-1, 1]; L^2_x (\T^2))} \\
&\les N_1^{s - \frac 23} \| \varphi_T u_{N_1} \|_{L^4_t ([-1, 1]; L^4_x (\T^2))} \| \varphi_T v_{N_2} \|_{L^4_t ([-1, 1]; L^4_x (\T^2))} \\
&\les N_1^{s - \frac 23} N_1^{4 \eps} \| \varphi_T u_{N_1} \|_{X^{0, \frac 12 - \eps}} N_2^{4 \eps} \| \varphi_T v_{N_2} \|_{X^{0, \frac 12 - \eps}} \\
&\les_\varphi N_1^{s - \frac 23} N_1^{-s + 4 \eps} T^{\frac{\eps}{2}} \| u_{N_1} \|_{X^{s, \frac 23}} N_2^{-s + 4 \eps} T^{\frac{\eps}{2}} \| v_{N_2} \|_{X^{s, \frac 23}} \\
&\les N_1^{-s - \frac 23 + 8 \eps} T^\eps \| u_{N_1} \|_{Z^{s, \frac 23}} \| v_{N_2} \|_{Z^{s, \frac 23}},
\end{align*}

\noi
where $\eps > 0$ is arbitrarily small. The above estimate is acceptable if $-s - \frac 23 + 8 \eps < 0$, which is valid given $s > -\frac 23$ and $\eps > 0$ small enough.

Regarding the $\l_n^2 L_\tau^1$ norm of the $\F_{x, t} \big( \varphi_T \cj{u_{N_1}} \cdot \varphi_T \cj{v_{N_2}} \big)$ term, we first let $\eps_1, \eps_2 > 0$ satisfying 
\begin{align*}
1 + \frac{1}{1 + \eps_1} = \frac{2}{1 + \eps_2}.
\end{align*}

\noi
Note that both $\eps_1$ and $\eps_2$ can be arbitrarily small. By H\"older's inequality, Young's convolution inequality, H\"older's inequalities twice, and Lemma \ref{LEM:time}, we have
\begin{align*}
\big\| &\jb{n}^s \jb{\tau + |n|^2}^{- 1} \F_{x, t} \big( \varphi_T \cj{u_{N_1}} \cdot \varphi_T \cj{v_{N_2}} \big) \big\|_{\l_n^2 L_\tau^1 (\Pf_N)} \\
&= \Big\| \jb{n}^s \jb{\tau + |n|^2}^{- 1} \wt{\ft{\varphi_T u_{N_1}}} * \wt{\ft{\varphi_T v_{N_2}}} \Big\|_{\l_n^2 L_\tau^1 (\Pf_N)} \\
&\les N_1^{- 2 + 2 \eps_1} \big\| \jb{n}^s \jb{\tau + |n|^2}^{- \eps_1} \big\|_{\l_n^2 L_\tau^{(1 + \eps_1) / \eps_1} (\Pf_N)} \big\| \ft{\varphi_T u_{N_1}} \big\|_{\l_{n_1}^2 L_{\tau_1}^{1 + \eps_2}} \big\| \ft{\varphi_T v_{N_2}} \big\|_{\l_{n_2}^2 L_{\tau_2}^{1 + \eps_2}} \\
&\les N_1^{- 2 + 2 \eps_1} N^{s + 1} \big\| \jb{\tau_1 + |n_1|^2}^{\frac{1}{2 + 2 \eps_1} +} \ft{\varphi_T u_{N_1}} \big\|_{\l_{n_1}^2 L_{\tau_1}^2} \big\| \jb{\tau_2 + |n_2|^2}^{\frac{1}{2 + 2 \eps_1} +} \ft{\varphi_T v_{N_2}} \big\|_{\l_{n_2}^2 L_{\tau_2}^2} \\
&\les_\varphi N^{s + 1} N_1^{- 2 + 2 \eps_1} T^\ta N_1^{-s} \| u_{N_1} \|_{X^{s, \frac 23}} T^{\ta} N_2^{-s} \| v_{N_2} \|_{X^{s, \frac 23}} \\
&\les N^{s + 1} N_1^{- 2s - 2 + 2 \eps_1} T^{2 \ta} \| u_{N_1} \|_{Z^{s, \frac 23}} \| v_{N_2} \|_{Z^{s, \frac 23}}
\end{align*}

\noi
for some $\ta > 0$. Since $s + 1 > 0$ given $s > - \frac 23$, the above estimate is acceptable if $-s - 1 + 2 \eps_1 < 0$, which is valid given $s > - \frac 23$ and $\eps_1 > 0$ close enough to 0. Combining the above two estimates, we obtain the desired inequality.

\medskip \noi
\textbf{Case 2:} $|\tau + |n|^2| < 2^{-10} |n_1|^2$, $|\tau_1 + |n_1|^2| \geq 2^{-10} |n_1|^2$, and $|\tau_2 + |n_2|^2| \geq 2^{-10} |n_2|^2$.

In this case, we need to estimate both $u_{N_1}$ and $v_{N_2}$ using the $Y^{s, \frac 23}$-norm. We consider the following two subcases.

\medskip \noi
\textbf{Subcase 2.1:} $|\tau + |n|^2| < 2^{-10} |n|^2$.

In this subcase, we need to evaluate the $\F_{x, t} \big( \varphi_T \cj{u_{N_1}} \cdot \varphi_T \cj{v_{N_2}} \big)$ term using the $\ft{X}^{s, \frac 23}$-norm. By H\"older's inequality, Young's convolution inequality, and Lemma \ref{LEM:time}, we have
\begin{align*}
\big\| &\jb{n}^s \jb{\tau + |n|^2}^{- \frac 13} \F_{x, t} \big( \varphi_T \cj{u_{N_1}} \cdot \varphi_T \cj{v_{N_2}} \big) \big\|_{\l_n^2 L_\tau^2 (\Pf_N)} \\
&= \Big\| \jb{n}^s \jb{\tau + |n|^2}^{- \frac 13} \wt{\ft{\varphi_T u_{N_1}}} * \wt{\ft{\varphi_T v_{N_2}}} \Big\|_{\l_n^2 L_\tau^2 (\Pf_N)} \\
&\les \big\| \jb{n}^s \jb{\tau + |n|^2}^{- \frac 13} \big\|_{\l_n^2 L_\tau^2 (\Pf_N)} \big\| \ft{\varphi_T u_{N_1}} \big\|_{\l_{n_1}^2 L_{\tau_1}^2} \big\| \ft{\varphi_T v_{N_2}} \big\|_{\l_{n_2}^2 L_{\tau_2}^2} \\
&\les_\varphi N^{s + 1} N^{\frac 13} T^\eps \big\| \jb{\tau_1 + |n_1|^2}^\eps \ft{u_{N_1}} \big\|_{\l_{n_1}^2 L_{\tau_1}^2} T^\eps \big\| \jb{\tau_2 + |n_2|^2}^\eps \ft{v_{N_2}} \big\|_{\l_{n_2}^2 L_{\tau_2}^2} \\
&\les N^{s + \frac 43} T^\eps N_1^{-s - \frac 43 + 2 \eps} \| u_{N_1} \|_{Y^{s, \frac 23}} T^\eps N_2^{-s - \frac 43 + 2 \eps} \| v_{N_2} \|_{Y^{s, \frac 23}} \\
&\les N^{s + \frac 43} N_1^{-2s - \frac 83 + 4 \eps} T^{2 \eps} \| u_{N_1} \|_{Z^{s, \frac 23}} \| v_{N_2} \|_{Z^{s, \frac 23}},
\end{align*}

\noi
where $\eps > 0$ is arbitrarily small. Since $s > -\frac 23$, we have $s + \frac 43 > 0$. Thus, the above estimate is acceptable if $- s - \frac 43 + 4 \eps < 0$, which is valid given $s > -\frac 23$.

\medskip \noi
\textbf{Subcase 2.2:} $2^{-10} |n|^2 \leq |\tau + |n|^2| < 2^{-10} |n_1|^2$.

In this subcase, we need to evaluate the $\F_{x, t} \big( \varphi_T \cj{u_{N_1}} \cdot \varphi_T \cj{v_{N_2}} \big)$ term using the $\ft{Y}^{s, \frac 23}$-norm. By H\"older's inequality, Young's convolution inequality, and Lemma \ref{LEM:time}, we have
\begin{align*}
\begin{split}
\big\| &\jb{\tau + |n|^2}^{\frac{s}{2} - \frac 13} \F_{x, t} \big( \varphi_T \cj{u_{N_1}} \cdot \varphi_T \cj{v_{N_2}} \big)  \big\|_{\l_n^2 L_\tau^2 (\Pf_N)} \\
&= \Big\| \jb{\tau + |n|^2}^{\frac{s}{2} - \frac 13} \wt{\ft{\varphi_T u_{N_1}}} * \wt{\ft{\varphi_T v_{N_2}}} \Big\|_{\l_n^2 L_\tau^2 (\Pf_N)} \\
&\les \big\| \jb{\tau + |n|^2}^{\frac{s}{2} - \frac 13} \big\|_{\l_n^2 L_\tau^2 (\Pf_N)} \big\| \ft{\varphi_T u_{N_1}} \big\|_{\l_{n_1}^2 L_{\tau_1}^2} \big\| \ft{ \varphi_T v_{N_2} } \big\|_{\l_{n_2}^2 L_{\tau_2}^2} \\
&\les_\varphi N^{s + \frac 13 + 2 \eps} \big\| \jb{\tau + |n|^2}^{-\frac 12 - \eps} \big\|_{\l_n^2 L_\tau^2 (\Pf_N)}  \\
&\qquad \times T^\eps \big\| \jb{\tau_1 + |n_1|^2}^\eps \ft{u_{N_1}} \big\|_{\l_{n_1}^2 L_{\tau_1}^2} T^\eps \big\| \jb{\tau_2 + |n_2|^2}^\eps \ft{v_{N_2}} \big\|_{\l_{n_2}^2 L_{\tau_2}^2} \\
&\les N^{s + \frac 43 + 2 \eps} T^{2 \eps} N_1^{-s - \frac 43 + 2\eps} \| u_{N_1} \|_{Y^{s, \frac 23}} N_2^{-s - \frac 43 + 2\eps} \| v_{N_2} \|_{Y^{s, \frac 23}} \\
&\les N^{s + \frac 43 + 2 \eps} N_1^{-2s - \frac 83 + 4 \eps} T^{2 \eps} \| u_{N_1} \|_{Z^{s, \frac 23}} \| v_{N_2} \|_{Z^{s, \frac 23}},
\end{split}
\end{align*}

\noi
where $\eps > 0$ is arbitrarily small. Note that the second inequality is valid since $s + \frac 13 + 2 \eps < 0$ given $s \leq - \frac 12$ and $\eps > 0$ small enough. Since $s + \frac 43 + 2 \eps > 0$ given $s > -\frac 23$, the above estimate is acceptable if $-s - \frac 43 + 6 \eps < 0$, which is valid given $s > -\frac 23$.

Also, by the Cauchy-Schwarz inequality, H\"older's inequality, Young's convolution inequality, and Lemma \ref{LEM:time}, we get
\begin{align*}
\big\| &\jb{n}^s \jb{\tau + |n|^2}^{- 1} \F_{x, t} \big( \varphi_T \cj{u_{N_1}} \cdot \varphi_T \cj{v_{N_2}} \big) \big\|_{\l_n^2 L_\tau^1 (\Pf_N)} \\
&\les \Big\| \jb{n}^s \jb{\tau + |n|^2}^{- \frac 12 + \eps} \wt{\ft{\varphi_T u_{N_1}}} * \wt{\ft{\varphi_T v_{N_2}}} \Big\|_{\l_n^2 L_\tau^2 (\Pf_N)} \\
&\les N^s N_1^{4 \eps} \big\| \jb{\tau + |n|^2}^{-\frac 12 - \eps} \big\|_{\l_n^2 L_\tau^2 (\Pf_N)} \big\| \ft{\varphi_T u_{N_1}} \big\|_{\l_{n_1}^2 L_{\tau_1}^2} \big\| \ft{ \varphi_T v_{N_2} } \big\|_{\l_{n_2}^2 L_{\tau_2}^2} \\
&\les_\varphi N^{s + 1} N_1^{4 \eps} T^\eps \big\| \jb{\tau_1 + |n_1|^2}^\eps \ft{u_{N_1}} \big\|_{\l_{n_1}^2 L_{\tau_1}^2} T^\eps \big\| \jb{\tau_2 + |n_2|^2}^\eps \ft{v_{N_2}} \big\|_{\l_{n_2}^2 L_{\tau_2}^2} \\
&\les N^{s + 1} N_1^{4 \eps} T^{2 \eps} N_1^{-s - \frac 43 + 2 \eps} \| u_{N_1} \|_{Y^{s, \frac 23}} N_2^{-s - \frac 43 + 2\eps} \| v_{N_2} \|_{Y^{s, \frac 23}} \\
&\les N^{s + 1} N_1^{-2s - \frac 83 + 8 \eps} T^{2 \eps} \| u_{N_1} \|_{Z^{s, \frac 23}} \| v_{N_2} \|_{Z^{s, \frac 23}},
\end{align*}

\noi
where $\eps > 0$ is arbitrarily small. Since $s + 1 > 0$ given $s > -\frac 23$, the above estimate is acceptable if $-s - \frac 53 + 8 \eps < 0$, which is valid given $s > -\frac 23$ and $\eps > 0$ small enough. Combining the above two estimates, we obtain the desired inequality.

\medskip \noi
\textbf{Case 3:} $|\tau + |n|^2| < 2^{-10} |n_1|^2$ and $|\tau_1 + |n_1|^2| < 2^{-10} |n_1|^2$.

In this case, we need to estimate $u_{N_1}$ using the $X^{s, \frac 23}$-norm. Note that we have
\begin{align*}
\tau_2 &= - \tau - \tau_1 \\
&= (-\tau - |n|^2) + (-\tau_1 - |n_1|^2) + |n|^2 + |n_1|^2 \\
&> - 2^{-10} |n_1|^2 - 2^{-10} |n_1|^2 + |n_1|^2 \\
&> 0,
\end{align*}

\noi
and so $|\tau_2 + |n_2|^2| > |n_2|^2 > 2^{-10} |n_2|^2$. Thus, we need to estimate $v_{N_2}$ using the $Y^{s, \frac 23}$-norm. We consider the following two subcases.

\medskip \noi
\textbf{Subcase 3.1:} $|\tau + |n|^2| < 2^{-10} |n|^2$.

In this subcase, we need to evaluate the $\F_{x, t} \big( \varphi_T \cj{u_{N_1}} \cdot \varphi_T \cj{v_{N_2}} \big)$ term using the $\ft{X}^{s, \frac 23}$-norm. By duality and the Cauchy-Schwarz inequality, we have
\begin{align}
\begin{split}
\big\| &\jb{n}^s \jb{\tau + |n|^2}^{- \frac 13} \F_{x, t} \big( \varphi_T \cj{u_{N_1}} \cdot \varphi_T \cj{v_{N_2}} \big) \big\|_{\l_n^2 L_\tau^2 ( \Pf_N )} \\
&\les N^s \sup_{\| h \|_{\l_n^2 L_\tau^2 ( \Pf_N )} \leq 1} \bigg| \sum_{\substack{n, n_1, n_2 \in \Z^2 \\ n + n_1 + n_2 = 0}} \iint_{\tau + \tau_1 + \tau_2 = 0} \ft{\varphi_T u_{N_1}} (n_1, \tau_1) \ft{\varphi_T v_{N_2}} (n_2, \tau_2) \\
&\qquad \times \frac{h (n, \tau)}{\jb{\tau + |n|^2}^{\frac 13}} d\tau d\tau_2 \bigg| \\
&\leq N^{s} \| \ft{\varphi_T v_{N_2} } \|_{\l_{n_2}^2 L_{\tau_2}^2} \sup_{\| h \|_{\l_n^2 L_\tau^2 ( \Pf_N )} \leq 1} \bigg\| \sum_{\substack{n, n_1 \in \Z^2 \\ n + n_1 + n_2 = 0}} \int_{\tau + \tau_1 + \tau_2 = 0} \ft{\varphi_T u_{N_1}} (n_1, \tau_1) \\
&\qquad \times \frac{h (n, \tau)}{\jb{\tau + |n|^2}^{\frac 13}} d\tau \bigg\|_{\l_{n_2}^2 L_{\tau_2}^2}.
\end{split}
\label{HH3-1}
\end{align}

\noi
Let $w_N$ be a space-time distribution that satisfy $\ft{w_N} (n, \tau) = h (n, \tau ) / \jb{\tau + |n|^2}^{\frac 13}$. Then, using the fact that $\varphi_T$ is supported on $[-1, 1]$ given $0 < T \leq \frac 12$, by the Plancherel theorem, H\"older's inequality, Lemma \ref{LEM:L4}, and Lemma \ref{LEM:time}, we have
\begin{align*}
\begin{split}
\bigg\| &\sum_{\substack{n, n_1 \in \Z^2 \\ n + n_1 + n_2 = 0}} \int_{\tau + \tau_1 + \tau_2 = 0} \ft{\varphi_T u_{N_1}} (n_1, \tau_1) \frac{h (n, \tau)}{\jb{\tau + |n|^2}^{\frac 13}} d\tau \bigg\|_{\l_{n_2}^2 L_{\tau_2}^2} \\
&= \| \varphi_T u_{N_1} \wt{w_N} \|_{L_t^2 ([-1, 1]; L_x^2 (\T^2))} \\
&\les \| \varphi_T u_{N_1} \|_{L_t^4 ([-1, 1]; L_x^4 (\T^2))} \| w_N \|_{L_t^4 ([-1, 1]; L_x^4 (\T^2))} \\
&\les N_1^{4\eps} \| \varphi_T u_{N_1} \|_{X^{0, \frac 12 - \eps}} N^{\frac 13 + \eps} \| w_N \|_{X^{0, \frac 13}} \\
&\les_\varphi N_1^{-s + 4 \eps} T^{\frac{\eps}{2}} \| u_{N_1} \|_{X^{s, \frac 23}} N^{\frac 13 + \eps} \| h \|_{\l_n^2 L_\tau^2 ( \Pf_N )},
\end{split}
\end{align*}

\noi
where $\eps > 0$ is arbitrarily small. Thus, continuing with \eqref{HH3-1}, we use Lemma \ref{LEM:time} to obtain
\begin{align*}
\big\| &\jb{n}^s \jb{\tau + |n|^2}^{- \frac 13} \F_{x, t} \big( \varphi_T \cj{u_{N_1}} \cdot \varphi_T \cj{v_{N_2}} \big) \big\|_{\l_n^2 L_\tau^2 ( \Pf_N )} \\
&\les_\varphi N^{s + \frac 13 + \eps} N_1^{-s + 4 \eps} T^{\frac{\eps}{2}} \| u_{N_1} \|_{X^{s, \frac 23}} \big\| \ft{\varphi_T v_{N_2}} \big\|_{\l_{n_2}^2 L_{\tau_2}^2} \\
&\les_\varphi N^{s + \frac 13 + \eps} N_1^{-s + 4 \eps} T^\eps \| u_{N_1} \|_{Z^{s, \frac 23}} \big\| \jb{\tau_2 + |n_2|^2}^{\frac{\eps}{2}} \ft{v_{N_2}} \big\|_{\l_{n_2}^2 L_{\tau_2}^2} \\
&\les N^{s + \frac 13 + \eps} N_1^{-s + 4 \eps} T^\eps \| u_{N_1} \|_{Z^{s, \frac 23}} N_2^{-s - \frac 43 + \eps} \| v_{N_2} \|_{Y^{s, \frac 23}} \\
&\les N^{s + \frac 13 + \eps} N_1^{-2s - \frac 43 + 5 \eps} T^\eps \| u_{N_1} \|_{Z^{s, \frac 23}} \| v_{N_2} \|_{Z^{s, \frac 23}}.
\end{align*}

\noi
Since $s \leq -\frac 12$, we have $s + \frac 13 + \eps < 0$ for $\eps > 0$ small enough. Thus, the above estimate is acceptable if $-2s - \frac 43 + 5 \eps \leq 0$, which is valid given $s > -\frac 23$ and $\eps > 0$ sufficiently small.

\medskip \noi
\textbf{Subcase 3.2:} $2^{-10}|n|^2 \leq |\tau + |n|^2| < 2^{-10} |n_1|^2$.

In this subcase, we need to evaluate the $\F_{x, t} \big( \varphi_T \cj{u_{N_1}} \cdot \varphi_T \cj{v_{N_2}} \big)$ term using the $Y^{s, \frac 23}$-norm. By duality and the Cauchy-Schwarz inequality, we have
\begin{align}
\begin{split}
\big\| &\jb{\tau + |n|^2}^{\frac{s}{2} - \frac 13} \F_{x, t} \big( \varphi_T \cj{u_{N_1}} \cdot \varphi_T \cj{v_{N_2}} \big) \big\|_{\l_n^2 L_\tau^2 (\Pf_N)} \\
&\les N^{s + \frac 13} \sup_{\| h \|_{\l_n^2 L_\tau^2 ( \Pf_N )} \leq 1} \bigg| \sum_{\substack{n, n_1, n_2 \in \Z^2 \\ n + n_1 + n_2 = 0}} \iint_{\tau + \tau_1 + \tau_2 = 0} \ft{\varphi_T u_{N_1}} (n_1, \tau_1) \ft{\varphi_T v_{N_2}} (n_2, \tau_2) \\
&\qquad \times \frac{h (n, \tau)}{\jb{\tau + |n|^2}^{\frac 12}} d\tau d\tau_2 \bigg| \\
&\leq N^{s + \frac 13} \| \ft{\varphi_T v_{N_2}} \|_{\l_{n_2}^2 L_{\tau_2}^2} \sup_{\| h \|_{\l_n^2 L_\tau^2 ( \Pf_N )} \leq 1} \bigg\| \sum_{\substack{n, n_1 \in \Z^2 \\ n + n_1 + n_2 = 0}} \int_{\tau + \tau_1 + \tau_2 = 0} \ft{\varphi_T u_{N_1}} (n_1, \tau_1) \\
&\qquad \times \frac{h (n, \tau)}{\jb{\tau + |n|^2}^{\frac 12}} d\tau \bigg\|_{\l_{n_2}^2 L_{\tau_2}^2}.
\end{split}
\label{HH3-2}
\end{align}

\noi
Note that the first inequality is valid since $s +  \frac 13 < 0$ given $s \leq - \frac 12$. Let $w_N$ be a space-time distribution that satisfy $\ft{w_N} (n, \tau) = h (n, \tau) / \jb{\tau + |n|^2}^{\frac 12}$. Then, using the fact that $\varphi_T$ is supported on $[-1, 1]$ given $0 < T \leq \frac 12$, by the Plancherel theorem, H\"older's inequality, Lemma \ref{LEM:L4}, and Lemma \ref{LEM:time}, we have
\begin{align*}
\bigg\| &\sum_{\substack{n, n_1 \in \Z^2 \\ n + n_1 + n_2 = 0}} \int_{\tau + \tau_1 + \tau_2 = 0} \ft{\varphi_T u_{N_1}} (n_1, \tau_1) \times \frac{h (n, \tau)}{\jb{\tau + |n|^2}^{\frac 12}} d\tau \bigg\|_{\l_{n_2}^2 L_{\tau_2}^2} \\
&= \| \varphi_T u_{N_1} \wt{w_N} \|_{L_t^2 ([-1, 1]; L_x^2 (\T^2))} \\
&\les \| \varphi_T u_{N_1} \|_{L_t^4 ([-1, 1]; L_x^4 (\T^2))} \| w_N \|_{L_t^4 ([-1, 1]; L_x^4 (\T^2))} \\
&\les N_1^{4 \eps} \| \varphi_T u_{N_1} \|_{X^{0, \frac 12 - \eps}} N^\eps \| w_N \|_{X^{0, \frac 12}} \\
&\les_\varphi N_1^{-s + 4 \eps} T^{\frac{\eps}{2}} \| u_{N_1} \|_{X^{s, \frac 23}} N^\eps \| h \|_{\l_n^2 L_\tau^2 (\Pf_N)},
\end{align*}

\noi
where $\eps > 0$ is arbitrarily small. Thus, continuing with \eqref{HH3-2}, we use Lemma \ref{LEM:time} to obtain
\begin{align*}
\big\| &\jb{\tau + |n|^2}^{\frac{s}{2} - \frac 13} \F_{x, t} \big( \varphi_T \cj{u_{N_1}} \cdot \varphi_T \cj{v_{N_2}} \big) \big\|_{\l_n^2 L_\tau^2 (\Pf_N)} \\
&\les_\varphi N^{s + \frac 13 + \eps} N_1^{-s + 4 \eps} T^{\frac{\eps}{2}} \| u_{N_1} \|_{X^{s, \frac 23}} \big\| \ft{\varphi_T v_{N_2}} \big\|_{\l_{n_2}^2 L_{\tau_2}^2} \\
&\les_\varphi N^{s + \frac 13 + \eps} N_1^{-s + 4 \eps} T^\eps \| u_{N_1} \|_{Z^{s, \frac 23}} \big\| \jb{\tau_2 + |n_2|^2}^{\frac{\eps}{2}} \ft{v_{N_2}} \big\|_{\l_{n_2}^2 L_{\tau_2}^2} \\
&\les N^{s + \frac 13 + \eps} N_1^{-s + 4 \eps} T^\eps \| u_{N_1} \|_{Z^{s, \frac 23}} N_2^{-s - \frac 43 + \eps} \| v_{N_2} \|_{Y^{s, \frac 23}} \\
&\les N^{s + \frac 13 + \eps} N_1^{-2s - \frac 43 + 5 \eps} T^\eps \| u_{N_1} \|_{Z^{s, \frac 23}} \| v_{N_2} \|_{Z^{s, \frac 23}}.
\end{align*}

\noi
Since $s \leq -\frac 12$, we have $s + \frac 13 + \eps < 0$ for $\eps > 0$ small enough. Thus, the above estimate is acceptable if $-2s - \frac 43 + 5 \eps \leq 0$, which is valid given $s > -\frac 23$ and $\eps > 0$ sufficiently small.

Also, by the Cauchy-Schwarz inequality, we get
\begin{align*}
\big\| &\jb{n}^s \jb{\tau + |n|^2}^{- 1} \F_{x, t} \big( \varphi_T \cj{u_{N_1}} \cdot \varphi_T \cj{v_{N_2}} \big) \big\|_{\l_n^2 L_\tau^1 (\Pf_N)} \\
&\les \big\| \jb{n}^s \jb{\tau + |n|^2}^{- \frac 12 + \eps} \F_{x, t} \big( \varphi_T \cj{u_{N_1}} \cdot \varphi_T \cj{v_{N_2}} \big) \big\|_{\l_n^2 L_\tau^2 (\Pf_N)},
\end{align*}

\noi
where $\eps > 0$ is arbitrarily small. The above term can be estimated similarly as above (along with $\jb{\tau + |n|^2}^{\eps} \les N_1^{2 \eps}$) for the $\l_n^2 L_\tau^2$ term. Combining the above two estimates, we obtain the desired inequality.

\medskip \noi
\textbf{Case 4:} $|\tau + |n|^2| < 2^{-10} |n_1|^2$ and $|\tau_2 + |n_2|^2| < 2^{-10} |n_2|^2$.

This case follows similarly from Case 3 by switching the roles of $u_{N_1}$ and $v_{N_2}$. We thus omit details.

\medskip
Thus, we have finished our proof.
\end{proof}

Before moving on to the proof of our main bilinear estimate in Proposition \ref{PROP:bilin}, we first observe that by definition of the $X^{s, b}$-norm in \eqref{Xsb} and the $Y^{s, b}$-norm in \eqref{Ysb}, we have the following decompositions:
\begin{align*}
\| u \|_{X^{s, b}}^2 &= \sum_{\substack{N \geq 1 \\ \text{dyadic}}} \| u_N \|_{X^{s, b}}^2, \\
\| u \|_{Y^{s, b}}^2 &= \sum_{\substack{N \geq 1 \\ \text{dyadic}}} \| u_N \|_{Y^{s, b}}^2.
\end{align*}

\noi
Thus, it follows that we have the following decomposition regarding the $Z^{s, b}$ norm:
\begin{align}
\begin{split}
\| u \|_{Z^{s, b}}^2 &\sim \| P_{\text{lo}} u \|_{X^{s, b}}^2 + \| P_{\text{hi}} u \|_{Y^{s, b}}^2 \\
&= \sum_{\substack{N \geq 1 \\ \text{dyadic}}} \big( \| P_{\text{lo}} u_N \|_{X^{s, b}}^2 + \| P_{\text{hi}} u_N \|_{Y^{s, b}}^2 \big) \\
&\sim \sum_{\substack{N \geq 1 \\ \text{dyadic}}} \| u_N \|_{Z^{s, b}}^2.
\end{split}
\label{dyadic}
\end{align}

\begin{proof}[Proof of Proposition \ref{PROP:bilin}]
By \eqref{dyadic}, we have
\begin{align}
\begin{split}
\big\|  &\jb{\tau + |n|^2}^{-1} \F_{x, t} \big( \varphi_T \cj{u} \cdot \varphi_T \cj{v} \big) (n, \tau)  \big\|_{\ft{Z}^{s, \frac 23}}^2 \\
&\les \sum_{\substack{N \geq 1 \\ \text{dyadic}}} \bigg( \sum_{\substack{N_1, N_2 \geq 1 \\ \text{dyadic}}} \big\| \jb{\tau + |n|^2}^{-1} \F_{x, t} \big( \varphi_T \cj{u_{N_1}} \cdot \varphi_T \cj{v_{N_2}} \big) (n, \tau) \big\|_{\ft{Z}^{s, \frac 23} (\Pf_N)} \bigg)^2.
\end{split}
\label{decomp}
\end{align}

\noi
For each nonzero summand on the right-hand side of \eqref{decomp}, we know that $N$, $N_1$, and $N_2$ must satisfy one of the following:
\begin{itemize}
\item[1.] $2^{-5} N \leq N_1 \leq 2^5 N$ and $N_2 \leq 2^6 N$, 

\item[2.] $2^{-5} N \leq N_2 \leq 2^5 N$ and $N_1 \leq 2^6 N$,

\item[3.] $\frac 12 N_1 \leq N_2 \leq 2N_1$ and $N < 2^{-5} N_1$.
\end{itemize}

\noi
We now treat the above three cases separately.

\medskip \noi
\textbf{Case 1:} $2^{-5} N \leq N_1 \leq 2^5 N$ and $N_2 \leq 2^6 N$.

In this case, by Lemma \ref{LEM:HL}, the Cauchy-Schwarz inequality, and \eqref{dyadic} twice, we have
\begin{align*}
\eqref{decomp} &\les \sum_{\substack{N \geq 1 \\ \text{dyadic}}} \bigg(  \sum_{\substack{N_1 \geq 1 \text{ dyadic} \\ 2^{-5}  \leq N_1/N \leq 2^5 }} \sum_{\substack{N_2 \geq 1 \text{ dyadic} \\ N_2 \leq 2^6 N}} N_2^{-\dl} T^\ta \| u_{N_1} \|_{Z^{s, \frac 23}} \| v_{N_2} \|_{Z^{s, \frac 23}}  \bigg)^2 \\
&\les T^{2 \ta} \sum_{\substack{N \geq 1 \\ \text{dyadic}}} \bigg(  \sum_{\substack{N_1 \geq 1 \text{ dyadic} \\ 2^{-5}  \leq N_1/N \leq 2^5 }} \| u_{N_1} \|_{Z^{s, \frac 23}} \bigg( \sum_{\substack{N_2 \geq 1 \\ \text{dyadic}}} \| v_{N_2} \|_{Z^{s, \frac 23}}^2 \bigg)^{1/2} \bigg)^2 \\
&\les T^{2 \ta} \sum_{\substack{N \geq 1 \\ \text{dyadic}}} \sum_{\substack{N_1 \geq 1 \text{ dyadic} \\ 2^{-5}  \leq N_1/N \leq 2^5 }} \| u_{N_1} \|_{Z^{s, \frac 23}}^2 \| v \|_{Z^{s, \frac 23}}^2 \\
&\les T^{2 \ta} \| u \|_{Z^{s, \frac 23}}^2 \| v \|_{Z^{s, \frac 23}}^2,
\end{align*}

\noi
where in the right-hand side of the first inequality we have $\dl > 0$.

\medskip \noi
\textbf{Case 2:} $2^{-5} N \leq N_2 \leq 2^5 N$ and $N_1 \leq 2^6 N$.

This case can be treated in the same way as Case 1, and so we omit details.

\medskip \noi
\textbf{Case 3:} $\frac 12 N_1 \leq N_2 \leq 2N_1$ and $N < 2^{-5} N_1$.

In this case, by Lemma \ref{LEM:HH}, the Cauchy-Schwarz inequality, and \eqref{dyadic} twice, we have
\begin{align*}
\eqref{decomp} &\les \sum_{\substack{N \geq 1 \\ \text{dyadic}}} \bigg( \sum_{\substack{N_1 \geq 1 \text{ dyadic} \\ N_1 > 2^5 N}} \sum_{\substack{N_2 \geq 1 \text{ dyadic} \\ 1/2 \leq N_2 / N_1 \leq 2}} N^{-\dl} T^\ta \| u_{N_1}  \|_{Z^{s, \frac 23}} \| v_{N_2} \|_{Z^{s, \frac 23}} \bigg)^2 \\
&\les T^{2 \ta} \sum_{\substack{N \geq 1 \\ \text{dyadic}}} N^{-2 \dl} \bigg( \sum_{\substack{N_1 \geq 1 \\ \text{dyadic}}} \| u_{N_1} \|_{Z^{s, \frac 23}}^2 \bigg) \sum_{\substack{N_1 \geq 1 \\ \text{dyadic}}} \bigg( \sum_{\substack{N_2 \geq 1 \text{ dyadic} \\ 1/2 \leq N_2 / N_1 \leq 2}}  \| v_{N_2} \|_{Z^{s, \frac 23}}  \bigg)^2 \\
&\les T^{2 \ta} \| u \|_{Z^{s, \frac 23}}^2  \sum_{\substack{N_1 \geq 1 \\ \text{dyadic}}} \| v_{N_1} \|_{Z^{s, \frac 23}}^2 \\
&\les T^{2 \ta} \| u \|_{Z^{s, \frac 23}}^2 \| v \|_{Z^{s, \frac 23}}^2,
\end{align*}

\noi
where in the right-hand side of the first inequality we have $\dl > 0$.

\medskip
Combining the above three cases, we have thus finished our proof.
\end{proof}

\section{Local well-posedness of the quadratic NLS}
\label{SEC:LWP}
In this section, we present the proof of Theorem \ref{THM:LWP}, local well-posedness of the quadratic NLS \eqref{qNLS} in the low regularity setting. As mentioned in Section \ref{SEC:intro}, we mainly focus our attention on local well-posedness of \eqref{qNLS} on $H^s (\T^2)$ for $- \frac 23 < s \leq -\frac 12$, using the estimates of the $Z^{s, b}$-norm in Section \ref{SEC:func} and Section \ref{SEC:bilin}.

By writing \eqref{qNLS} in the Duhamel formulation, we have
\begin{align}
u(t) = \G [u] (t) := e^{it \Dl} u_0 - i  \int_0^t e^{i (t - t') \Dl} \cj{u}^2 (t') dt'.
\label{Duh}
\end{align}

\noi
Since we are only interested in local well-posedness, we can insert time cut-off functions. For $0 < T \leq \frac 12$, we let $\eta : \R \to [0, 1]$ be a smooth function such that $\eta \equiv 1$ on $[-1, 1]$ and $\eta \equiv 0$ outside of $[-2, 2]$ and let $\eta_{2T} (t) := \eta (t / 2T)$. We first replace the two $\cj{u}$'s on the right-hand side of \eqref{Duh} by $\eta_{2T} \cj{u}$. Also, note that for any function $F$ that is smooth in space and Schwartz in time, we have
\begin{align*}
\int_0^t e^{i (t - t') \Dl} F(x, t') dt' &= \int_0^t \sum_{n \in \Z^2} e^{in \cdot x} e^{-i (t - t') |n|^2} \int_\R  e^{i t' \tau} \ft{F} (n, \tau) d\tau dt' \\
&= \sum_{n \in \Z^2} e^{i n \cdot x} \int_\R e^{-it |n|^2} \ft{F} (n, \tau) \frac{e^{it (\tau + |n|^2)} - 1}{i (\tau + |n|^2)} d\tau \\
&= \sum_{n \in \Z^2} e^{i n \cdot x} \int_\R e^{-it |n|^2} \ft{F} (n, \tau) \psi (\tau + |n|^2) \frac{e^{it (\tau + |n|^2)} - 1}{i (\tau + |n|^2)} d\tau \\
&\qquad - \sum_{n \in \Z^2} e^{i n \cdot x} \int_\R e^{-it |n|^2} \ft{F} (n, \tau) \frac{1 - \psi (\tau + |n|^2)}{i(\tau + |n|^2)} d\tau \\
&\qquad + \sum_{n \in \Z^2} e^{i n \cdot x} \int_\R e^{it \tau} \ft{F} (n, \tau) \frac{1 - \psi (\tau + |n|^2)}{i(\tau + |n|^2)} d\tau,
\end{align*}

\noi
where $\psi: \R \to [0,1]$ is a smooth cut-off function such that $\psi \equiv 1$ on $[-1, 1]$ and $\psi \equiv 0$ outside of $[-2, 2]$. Let us define the following nonlinear terms. 
\begin{align}
\begin{split}
\mathcal{N}_1 (u, v) &:= -i \eta (t)  \sum_{n \in \Z^2} e^{i n \cdot x} \int_\R e^{-it |n|^2} \F_{x, t} \big( \eta_{2T} \cj{u} \cdot \eta_{2T} \cj{v} \big) (n, \tau)  \\
&\qquad \times \psi (\tau + |n|^2) \frac{e^{it (\tau + |n|^2)} - 1}{i (\tau + |n|^2)} d\tau, \\
\mathcal{N}_2 (u, v) &:= i \eta (t) \sum_{n \in \Z^2} e^{i n \cdot x} \int_\R e^{-it |n|^2} \F_{x, t} \big( \eta_{2T} \cj{u} \cdot \eta_{2T} \cj{v} \big) (n, \tau) \frac{1 - \psi (\tau + |n|^2)}{i(\tau + |n|^2)} d\tau, \\
\mathcal{N}_3 (u, v) &:= -i \sum_{n \in \Z^2} e^{i n \cdot x} \int_\R e^{it \tau} \F_{x, t} \big( \eta_{2T} \cj{u} \cdot \eta_{2T} \cj{v} \big) (n, \tau) \frac{1 - \psi (\tau + |n|^2)}{i(\tau + |n|^2)} d\tau.
\end{split}
\label{N123}
\end{align}

\noi
We consider the following formulation of the quadratic NLS \eqref{qNLS}:
\begin{align}
u(t) = \G_1 [u] (t) := \eta (t) e^{it \Dl} u_0 + \mathcal{N}_1 (u, u) + \mathcal{N}_2 (u, u) + \mathcal{N}_3 (u, u).
\label{Duh1}
\end{align}

\subsection{Relevant estimates}
\label{SUBSEC:est}
In this subsection, we present some relevant estimates for proving our local well-posedness result. We first show the following homogeneous linear estimate.
\begin{lemma}
\label{LEM:lin1}
Let $s \in \R$, $b \in \R$, and $0 < T \leq 1$. Then, we have
\begin{align*}
\big\| \eta (t) e^{it \Dl} \phi \big\|_{Z_T^{s, b}} \les_\eta \| \phi \|_{H^s (\T^2)},
\end{align*}
\end{lemma}

\begin{proof}
By the definition of the $Z_T^{s, b}$-norm in \eqref{ZsbT}, Lemma \ref{LEM:est2}, and Lemma \ref{LEM:lin_tk} with $k = 0$, we have
\begin{align*}
\| \eta (t) e^{it \Dl} \phi \|_{Z_T^{s, b}} \leq \| \eta (t) e^{it \Dl} \phi \|_{Z^{s, b}} \les \| \eta (t) e^{it \Dl} \phi \|_{X^{s, b}} \les_\eta \| \phi \|_{H^s (\T^2)},
\end{align*}

\noi
as desired.
\end{proof}

We now take $b = \frac 23$ and show the following bilinear estimate.
\begin{lemma}
\label{LEM:N123}
Let $-\frac 23 < s \leq -\frac 12$ and $0 < T \leq \frac 14$. Then, we have
\begin{align*}
\| \mathcal{N}_1 (u, v) \|_{Z_T^{s, \frac 23}} &\les_\eta T^\ta \| u \|_{Z_T^{s, \frac 23}} \| v \|_{Z_T^{s, \frac 23}}, \\
\| \mathcal{N}_2 (u, v) \|_{Z_T^{s, \frac 23}} &\les_\eta T^\ta \| u \|_{Z_T^{s, \frac 23}} \| v \|_{Z_T^{s, \frac 23}}, \\
\| \mathcal{N}_3 (u, v) \|_{Z_T^{s, \frac 23}} &\les_\eta T^\ta \| u \|_{Z_T^{s, \frac 23}} \| v \|_{Z_T^{s, \frac 23}}
\end{align*}

\noi
for some $\ta > 0$, where $\mathcal{N}_1$, $\mathcal{N}_2$, and $\mathcal{N}_3$ are as defined in \eqref{N123}.
\end{lemma}

\begin{proof}
The idea of the proof comes from \cite{Bour93}. As in the proof of Lemma \ref{LEM:embed}, by working with the extensions of $u$ and $v$ outside $[-T, T]$, it suffices to show the following three estimates:
\begin{align*}
\| \mathcal{N}_1 (u, v) \|_{Z^{s, \frac 23}} &\les_\eta T^\ta \| u \|_{Z^{s, \frac 23}} \| v \|_{Z^{s, \frac 23}}, \\
\| \mathcal{N}_2 (u, v) \|_{Z^{s, \frac 23}} &\les_\eta T^\ta \| u \|_{Z^{s, \frac 23}} \| v \|_{Z^{s, \frac 23}}, \\
\| \mathcal{N}_3 (u, v) \|_{Z^{s, \frac 23}} &\les_\eta T^\ta \| u \|_{Z^{s, \frac 23}} \| v \|_{Z^{s, \frac 23}},
\end{align*}

\noi
for some $\ta > 0$.

To deal with the $\mathcal{N}_1$ term, by Lemma \ref{LEM:est2}, the Taylor expansion, Lemma \ref{LEM:lin_tk}, Lemma \ref{LEM:est1}, and Proposition \ref{PROP:bilin}, we obtain
\begin{align*}
\begin{split}
\| \mathcal{N}_1 (u, v) \|_{Z^{s, \frac 23}}  &\les \| \mathcal{N}_1 (u, v) \|_{X^{s, \frac 23}} \\
&\les \sum_{k = 1}^\infty \frac{1}{k!} \bigg\| t^k \eta (t) e^{it \Dl} \sum_{n \in \Z^2} e^{in \cdot x} \int_\R \F_{x, t} \big( \eta_{2T} \cj{u} \cdot \eta_{2T} \cj{v} \big) (n, \tau) \\
&\qquad \times \psi (\tau + |n|^2) \jb{\tau + |n|^2}^{k - 1} d\tau \bigg\|_{X^{s, \frac 23}} \\
&\les_\eta \sum_{k = 1}^\infty \frac{3^k}{k!} \bigg\| \jb{n}^s \int_\R \F_{x, t} \big( \eta_{2T} \cj{u} \cdot \eta_{2T} \cj{v} \big) (n, \tau) \psi (\tau + |n|^2) \jb{\tau + |n|^2}^{k - 1} d\tau \bigg\|_{\l_n^2} \\
&\les \sum_{k = 1}^\infty \frac{6^k}{k!} \big\| \jb{n}^s \jb{\tau + |n|^2}^{-1} \F_{x, t} \big( \eta_{2T} \cj{u} \cdot \eta_{2T} \cj{v} \big) (n, \tau) \big\|_{\l_n^2 L_\tau^1} \\
&\les \big\| \jb{\tau + |n|^2}^{-1} \F_{x, t} \big( \eta_{2T} \cj{u} \cdot \eta_{2T} \cj{v} \big) (n, \tau) \big\|_{\ft{Z}^{s, \frac 23}} \\
&\les_\eta T^\ta \| u \|_{Z^{s, \frac 23}} \| v \|_{Z^{s, \frac 23}}
\end{split}
\end{align*}

\noi
for some $\ta > 0$.

For the $\mathcal{N}_2$ term, using Lemma \ref{LEM:est2}, Lemma \ref{LEM:lin_tk} with $k = 0$, the fact that $1 - \psi$ is bounded by 1 and supported outside of $[-1, 1]$, Lemma \ref{LEM:est1}, and Proposition \ref{PROP:bilin}, we have
\begin{align*}
\begin{split}
\| \mathcal{N}_2 (u, v) \|_{Z^{s, \frac 23}} &\les \| \mathcal{N}_2 (u, v) \|_{X^{s, \frac 23}} \\
&\les \bigg\| \eta(t) e^{it \Dl} \sum_{n \in \Z^2} e^{in \cdot x} \int_\R \F_{x, t} \big( \eta_{2T} \cj{u} \cdot \eta_{2T} \cj{v} \big) (n, \tau) \\
&\qquad \times \frac{1 - \psi (\tau + |n|^2)}{i (\tau + |n|^2)} d\tau \bigg\|_{X^{s, \frac 23 }} \\
&\les_\eta \bigg\| \jb{n}^s \int_\R \F_{x, t} \big( \eta_{2T} \cj{u} \cdot \eta_{2T} \cj{v} \big) (n, \tau)  \frac{1 - \psi (\tau + |n|^2)}{i (\tau + |n|^2)} d\tau \bigg\|_{\l_n^2} \\
&\les \big\| \jb{n}^s \jb{\tau + |n|^2}^{-1} \F_{x, t} \big( \eta_{2T} \cj{u} \cdot \eta_{2T} \cj{v} \big) (n, \tau) \big\|_{\l_n^2 L_\tau^1} \\
&\les \big\| \jb{\tau + |n|^2}^{-1} \F_{x, t} \big( \eta_{2T} \cj{u} \cdot \eta_{2T} \cj{v} \big) (n, \tau) \big\|_{\ft{Z}^{s, \frac 23}} \\
&\les_\eta T^\ta \| u \|_{Z^{s, \frac 23}} \| v \|_{Z^{s, \frac 23}}
\end{split}
\end{align*}

\noi
for some $\ta > 0$.

For the $\mathcal{N}_3$ term, since $1 - \psi$ is bounded by 1 and supported outside of $[-1, 1]$, by the monotonicity property \eqref{mono} and Proposition \ref{PROP:bilin}, we have
\begin{align*}
\begin{split}
\| \mathcal{N}_3 (u, v) \|_{Z^{s, \frac 23}} &= \bigg\| \F_{x, t} \big( \eta_{2T} \cj{u} \cdot \eta_{2T} \cj{v} \big) (n, \tau) \frac{1 - \psi (\tau + |n|^2)}{\tau + |n|^2} \bigg\|_{\ft{Z}^{s, \frac 23}} \\
&\les \big\| \jb{\tau + |n|^2}^{-1} \F_{x, t} \big( \eta_{2T} \cj{u} \cdot \eta_{2T} \cj{v} \big) \big\|_{\ft{Z}^{s, \frac 23}} \\
&\les_\eta T^\ta \| u \|_{Z^{s, \frac 23}} \| v \|_{Z^{s, \frac 23}}
\end{split}
\end{align*}

\noi
for some $\ta > 0$. Thus, we finish our proof.
\end{proof}

\subsection{Local well-posedness}
We now use the formulation \eqref{Duh1} and the estimates in Subsection \ref{SUBSEC:est} to prove our local well-posedness result. We let $0 < T \leq \frac 14$ and fix $-\frac 23 < s \leq - \frac 12$.

For the setting of $\T^2$, by \eqref{Duh1}, Lemma \ref{LEM:lin1}, and Lemma \ref{LEM:N123}, we have
\begin{align}
\begin{split}
\big\| \G_1 [u] \big\|_{Z_T^{s, \frac 23}} &\les \big\| \eta(t) e^{it \Dl} u_0 \big\|_{Z_T^{s, \frac 23}} + \sum_{j = 1}^3 \| \mathcal{N}_j (u, u) \|_{Z_T^{s, \frac 23}} \\
&\les_\eta \| u_0 \|_{H^s (\T^2)} + T^\ta \| u \|_{Z_T^{s, \frac 23}}^2,
\end{split}
\label{ball}
\end{align}

\noi
for some $\ta > 0$. Similarly, we obtain the following difference estimate:
\begin{align}
\begin{split}
\big\| \G_1 [u] - \G_1 [v] \big\|_{Z_T^{s, \frac 23}} &\les \sum_{j = 1}^3 \Big( \| \mathcal{N}_j (u, u - v) \|_{Z_T^{s, \frac 23}} + \| \mathcal{N}_j (u - v, v) \|_{Z_T^{s, \frac 23}} \Big) \\
&\les_\eta T^\ta \Big( \| u \|_{Z_T^{s, \frac 23}} + \| v \|_{Z_T^{s, \frac 23}} \Big) \| u - v \|_{Z_T^{s, \frac 23}}.
\end{split}
\label{diff}
\end{align}

\noi
Thus, by choosing $T = T(\| u_0 \|_{H^s (\T^2)}) > 0$ sufficiently small, we have that $\G_1$ is a contraction on the ball $B_R \subset Z^{s, \frac 23}$ of radius $R \sim \| u_0 \|_{H^s (\T^2)}$.  This gives the existence part of Theorem \ref{THM:LWP} when $\mathcal{M} = \T^2$ and the uniqueness in the ball $B_R$. Also, the continuous dependence of solutions on the initial data follows easily from the formulation \eqref{Duh1}, Lemma \ref{LEM:lin1}, \eqref{ball}, and \eqref{diff}.

It remains to extend the uniqueness of solutions to \eqref{qNLS} to the entire $Z_T^{s, \frac 23}$-space. We let $u$ and $v$ be two solutions of \eqref{qNLS} in $Z_T^{s, \frac 23}$. Note that $u$ and $v$ satisfy the formulation \eqref{Duh1} for $t \in [-T, T]$. For $0 < T_0 \leq T$, we use \eqref{diff} to obtain
\begin{align*}
\| u - v \|_{Z_{T_0}^{s, \frac 23}} &\les_\eta T_0^\ta \Big( \| u \|_{Z_{T_0}^{s, \frac 23}} + \| v \|_{Z_{T_0}^{s, \frac 23}} \Big) \| u - v \|_{Z_{T_0}^{s, \frac 23}} \\
&\leq T_0^\ta \Big( \| u \|_{Z_T^{s, \frac 23}} + \| v \|_{Z_T^{s, \frac 23}} \Big) \| u - v \|_{Z_{T_0}^{s, \frac 23}}.
\end{align*}

\noi
Thus, by choosing 
\begin{align*}
T_0 = T_0 \Big( \| u \|_{Z_T^{s, \frac 23}}, \| v \|_{Z_T^{s, \frac 23}} \Big) > 0
\end{align*} 
sufficiently small, we can use Lemma \ref{LEM:embed} to obtain
\begin{align*}
\| u - v \|_{C ( [-T_0, T_0] ); H^s (\T^2)} \les \| u - v \|_{Z_{T_0}^{s, \frac 23}} = 0,
\end{align*}

\noi
so that $u \equiv v$ on $[-T_0, T_0]$. Since $T_0$ depends only on $\| u \|_{Z_T^{s, \frac 23}}$ and $\| v \|_{Z_T^{s, \frac 23}}$, we can iterate the above argument on $[-T, -T_0]$ and $[T_0, T]$. This shows that $u \equiv v$ on $[-T, T]$ after a finite number of iterations, and so the uniqueness of \eqref{qNLS} on the entire $Z_T^{s, \frac 23}$-space follows.

\begin{ackno}\rm
The author would like to thank his advisor, Tadahiro Oh, for suggesting this problem and for his support throughout the entire work. The author is also grateful to the anonymous reviewers for the helpful comments. R.L. was supported by the European Research Council (grant no. 864138 ``SingStochDispDyn'').
\end{ackno}

\medskip \noi
\textbf{Declarations.} \\
Ethical approval: \\
Not applicable.

\smallskip \noi
Competing interests: \\
The author has no competing interests to declare.

\smallskip \noi
Author's contributions: \\
Not applicable.

\smallskip \noi
Funding: \\
The author acknowledges funding from the European Research Council (grant no. 864138 ``SingStochDispDyn'').

\smallskip \noi
Availability of data and materials: \\
Not applicable.


\begin{thebibliography}{99}

\bibitem{BS}
I.~Bejenaru, D.~de Silva,
{\it Low regularity solutions for a 2D quadratic nonlinear Schr\"odinger equation,}
Trans. Amer. Math. Soc. 360 (2008), no. 11, 5805--5830.

\bibitem{BT}
I.~Bejenaru, T.~Tao,
{\it Sharp well-posedness and ill-posedness results for a quadratic non-linear Schr\"odinger equation},
J. Funct. Anal. 233 (2006), no. 1, 228-259.

\bibitem{Bour93}
J.~Bourgain,
{\it Fourier transform restriction phenomena for certain lattice subsets and applications to nonlinear evolution equations, I: Schr\"odinger equations}, Geom. Funct. Anal. 3 (1993), 107--156.

\bibitem{Bour95}
J.~Bourgain, 
{\it Nonlinear Schr\"odinger equations,} Hyperbolic equations and frequency interactions (Park City, UT, 1995), 3--157,
IAS/Park City Math. Ser., 5, Amer. Math. Soc., Providence, RI, 1999.

\bibitem{CDKS}
J.~Colliander, J.~Delort, C.~Kenig, G.~Staffilani,
{\it Bilinear estimates and applications to 2D NLS},
Trans. Amer. Math. Soc. 353 (2001), no. 8, 3307--3325.

\bibitem{CKSTT04}
J.~Colliander, M.~Keel, G.~Staffilani, H.~Takaoka, T.~Tao,
{\it Multilinear estimates for periodic KdV equations, and applications,}
J. Funct. Anal. 211 (2004), no. 1, 173--218.


\bibitem{FG}
K.~Fujiwara, V.~Georgiev,
{\it On global existence of $L^2$ solutions for 1D periodic NLS with quadratic nonlinearity},
J. Math. Phys. 62 (2021), no. 9, Paper No. 091504, 9 pp.

\bibitem{GMS}
P.~Germain, N.~Masmoudi, J.~Shatah,
{\it Global solutions for 2D quadratic Schr\"odinger equations},
J. Math. Pures Appl. (9) 97 (2012), no. 5, 505--543.

\bibitem{Grun}
A.~Gr\"unrock,
{\it Some local wellposedness results for nonlinear Schr\"odinger equations below $L^2$,}
arXiv:math/0011157v2 [math.AP].

\bibitem{Guo}
Z.~Guo,
{\it Global well-posedness of Korteweg-de Vries equation in $H^{-3/4}(\R)$},
J. Math. Pures Appl. (9) 91 (2009), no. 6, 583--597.

\bibitem{HNST}
N.~Hayashi, P.~Naumkin, A.~Shimomura, S.~Tonegawa,
{\it Modified wave operators for nonlinear Schr\"odinger equations in one and two dimensions},
Electron. J. Differential Equations 2004, No. 62, 16 pp.

\bibitem{IW}
M.~Ikeda, Y.~Wakasugi,
{\it Small-data blow-up of $L^2$-solution for the nonlinear Schr\"odinger equation without gauge invariance},
Differential Integral Equations 26 (2013), no. 11-12, 1275--1285.

\bibitem{IO}
T.~Iwabuchi, T.~Ogawa,
{\it Ill-posedness for the nonlinear Schr\"odinger equation with quadratic non-linearity in low dimensions},
Trans. Amer. Math. Soc. 367 (2015), no. 4, 2613--2630.

\bibitem{IU}
T.~Iwabuchi, K.~Uriya,
{\it Ill-posedness for the quadratic nonlinear Schr\"odinger equation with nonlinearity $|u|^2$},
Commun. Pure Appl. Anal. 14 (2015), no. 4, 1395--1405.

\bibitem{JLT}
J.~Jaquette, J.~Lessard, A.~Takayasu,
{\it Global dynamics in nonconservative nonlinear Schr\"odinger equations},
Adv. Math. 398 (2022), Paper No. 108234, 70 pp.

\bibitem{KPV96}
C.~Kenig, G.~Ponce, L.~Vega,
{\it Quadratic forms for the 1-D semilinear Schr\"odinger equation},
Trans. Amer. Math. Soc. 348 (1996), no. 8, 3323--3353.

\bibitem{Kish08}
N.~Kishimoto,
{\it Local well-posedness for the Cauchy problem of the quadratic Schr\"odinger equation with nonlinearity $\cj{u}^2$},
Commun. Pure Appl. Anal. 7 (2008), no. 5, 1123--1143.

\bibitem{Kishm}
N.~Kishimoto,
{\it Low-regularity local well-posedness for quadratic nonlinear Schr\"odinger equations}, Kyoto University, master's thesis, 2008.

\bibitem{Kish09}
N.~Kishimoto,
{\it Low-regularity bilinear estimates for a quadratic nonlinear Schr\"odinger equation},
J. Differential Equations 247 (2009), no. 5, 1397--1439.

\bibitem{Kish09k}
N.~Kishimoto,
{\it Well-posedness of the Cauchy problem for the Korteweg-de Vries equation at the critical regularity},
Differential Integral Equations 22 (2009), no. 5-6, 447--464.

\bibitem{Kish19}
N.~Kishimoto,
{\it A remark on norm inflation for nonlinear Schr\"odinger equations},
Commun. Pure Appl. Anal. 18 (2019), no. 3, 1375--1402.

\bibitem{KT}
N.~Kishimoto, K.~Tsugawa,
{\it Local well-posedness for quadratic nonlinear Schr\"odinger equations and the ``good'' Boussinesq equation},
Differential Integral Equations 23 (2010), no. 5-6, 463--493.


\bibitem{LO}
R.~Liu, T.~Oh,
{\it Sharp local well-posedness of the two-dimensional periodic nonlinear Schr\"odinger equation with a quadratic nonlinearity $|u|^2$}, to appear in Math. Res. Lett.

\bibitem{MTT}
K.~Moriyama, S.~Tonegawa, Y.~Tsutsumi,
{\it Wave operators for the nonlinear Schr\"odinger equation with a nonlinearity of low degree in one or two space dimensions},
Commun. Contemp. Math. 5 (2003), no. 6, 983--996.

\bibitem{NTT}
K.~Nakanishi, H.~Takaoka, Y.~Tsutsumi,
{\it Counterexamples to bilinear estimates related with the KdV equation and the nonlinear Schr\"odinger equation,}
Methods Appl. Anal. 8 (2001), no. 4, 569--578.

\bibitem{Oh}
T.~Oh,
{\it A blowup result for the periodic NLS without gauge invariance},
C. R. Math. Acad. Sci. Paris 350 (2012), no. 7-8, 389--392.

\bibitem{Shi}
A.~Shimomura,
{\it Nonexistence of asymptotically free solutions for quadratic nonlinear Schr\"odinger equations in two space dimensions},
Differential Integral Equations 18 (2005), no. 3, 325--335.

\bibitem{ST}
A.~Shimomura, S.~Tonegawa,
{\it Long-range scattering for nonlinear Schr\"odinger equations in one and two space dimensions},
Differential Integral Equations 17 (2004), no. 1--2, 127--150.

\bibitem{STsu}
A.~Shimomura, Y.~Tsutsumi,
{\it Nonexistence of scattering states for some quadratic nonlinear Schr\"odinger equations in two space dimensions},
Differential Integral Equations 19 (2006), no. 9, 1047--1060.

\bibitem{Sta97}
G.~Staffilani,
{\it Quadratic forms for a 2-D semilinear Schr\"odinger equation},
Duke Math. J. 86 (1997), no. 1, 79--107.

\bibitem{Tao}
T.~Tao, 
{\it Nonlinear dispersive equations. Local and global analysis,}
CBMS Regional Conference Series in Mathematics, 106. Published for the Conference Board of the Mathematical Sciences, Washington, DC; by the American Mathematical Society, Providence, RI, 2006. xvi+373 pp.

\bibitem{Zyg}
A.~Zygmund,
{\it On Fourier coefficients and transforms of functions of two variables,}
Studia Math. 50 (1974), 189--201.

\end{thebibliography}
\end{document}